\documentclass[11pt]{amsart}
\usepackage{graphicx, amssymb, amsmath, amsthm}
\numberwithin{equation}{section}

\usepackage{amsmath,amsfonts}
\usepackage{latexsym}
\usepackage{amssymb,leftidx}
\usepackage{extarrows}
\usepackage{overpic}
\usepackage{color}
\usepackage{epsfig}
\usepackage{subfigure}
\usepackage{tikz}
\usetikzlibrary{positioning}
\usepackage{caption}

 \usepackage[colorlinks=true]{hyperref}
\hypersetup{urlcolor=blue, citecolor=red}

\usepackage{amssymb}
\usepackage{mathrsfs}
\usepackage{amscd}
\usepackage{cite}

\newcommand\specl{dE_{\sqrt{H}}(\lambda)}

\newcommand\R{\mathbb{R}}

\newcommand\Z{\mathbb{Z}}

\numberwithin{equation}{section}
\newtheorem{proposition}{Proposition}[section]

\newtheorem{lemma}{Lemma}[section]
\newtheorem{theorem}{Theorem}[section]

\newtheorem{remark}{Remark}[section]

\begin{document}

\author{Yannick Sire}
\address{Department of Mathematics\\                Johns Hopkins University\\                Baltimore, MD 21218}
\email{sire@math.jhu.edu}

\author{Christopher D. Sogge}
\address{Department of Mathematics\\                Johns Hopkins University\\                Baltimore, MD 21218}
\email{sogge@math.jhu.edu}

\author{Chengbo Wang}
\address{School of Mathematical Sciences\\                Zhejiang University\\                Hangzhou 310027, China}
\email{wangcbo@zju.edu.cn}

\author{Junyong Zhang}
\address{Department of Mathematics\\  Beijing Institute of Technology\\ Beijing 100081, China;  Cardiff University, UK}
\email{zhang\_junyong@bit.edu.cn; ZhangJ107@cardiff.ac.uk}

\email{}

\title[Strichartz estimates and Strauss conjecture on AH ]{Strichartz estimates and Strauss conjecture on non-trapping asymptotically hyperbolic manifolds}
\maketitle

\begin{abstract} We prove global-in-time Strichartz estimates for the shifted wave equations on non-trapping
asymptotically hyperbolic manifolds. The key tools are the spectral measure estimates from
\cite{CH2} and arguments borrowed from \cite{HZ, Zhang}. As an application, we prove  
the small data global existence for any power $p\in(1, 1+\frac{4}{n-1})$ for the shifted wave equation in this setting, 
involving nonlinearities of the form $\pm|u|^p$ or $\pm|u|^{p-1}u$, which answers partially an open question raised  in \cite{SSW}.
\end{abstract}

\tableofcontents
\begin{center}
 \begin{minipage}{120mm}
 \end{minipage}
 \end{center}
\section{Introduction and Main Results}
The purpose of this paper is to study the dispersive behaviour of the linear wave equation on non-trapping asymptotically hyperbolic manifolds, which is a class of manifolds with variable curvature, and its application to the  small data global existence for the nonlinear Cauchy problem with power nonlinearities. 

\subsection{Background on Strichartz estimates}The dispersive decay and Strichartz estimates are known to play an important role in the study of the behaviour of solutions to nonlinear Schr\"odinger equation, nonlinear wave equations and other nonlinear dispersive equations; e.g. see Tao \cite{Tao}. The first aim of this article is to prove global-in-time Strichartz estimates for the wave equation on non-trapping asymptotically hyperbolic manifolds.\vspace{0.2cm}

Let $(M^\circ, g)$ be a Riemannian manifold of dimension $n\geq2$,
and let $I\subset\R$ be a time interval. Suppose $u(t,z)$: $I\times
M^\circ\rightarrow\mathbb{R}$ is a  solution of the wave
equation
\begin{equation*}
\partial_{t}^2u-\Delta_g u=F, \quad u(0)=u_0(z), ~\partial_tu(0)=u_1(z),
\end{equation*}
where $\Delta_g$ denotes the Laplace-Beltrami
operator on $(M^\circ, g)$.
The general Strichartz
estimates show
\begin{equation}\label{stri-e}
\begin{split}
&\|u(t,z)\|_{L^q_t(I;L^r_z(M^\circ))}+\|u(t,z)\|_{C(I;\dot H^s(M^\circ))}\\
&\qquad\lesssim \|u_0\|_{\dot H^s(M^\circ)}+\|u_1\|_{\dot
H^{s-1}(M^\circ)}+\|F\|_{L^{\tilde{q}'}_t(I;L^{\tilde{r}'}_z(M^\circ))},
\end{split}
\end{equation}
where $\dot H^s$ denotes the homogeneous $L^2$-Sobolev space over $M^\circ$, and the pairs $(q,r), (\tilde{q},\tilde{r})\in [2,\infty]^2$
satisfy the wave-admissible condition
\begin{equation}\label{adm}
\frac{2}q+\frac{n-1}r\leq\frac{n-1}2,\quad (q,r,n)\neq(2,\infty,3)
\end{equation}
and the gap condition
\begin{equation}\label{scaling}
\frac1q+\frac nr=\frac n2-s=\frac1{\tilde{q}'}+\frac
n{\tilde{r}'}-2.
\end{equation}
It is well known that \eqref{stri-e} holds for $(M^\circ, g)=(\R^n,\delta)$ with $I=\R$ and $r,\tilde{r}<\infty$, and the result is sharp; see Strichartz
\cite{Str}, Ginibre-Velo \cite{GV}, Keel-Tao \cite{KT}, and
references therein. There is a huge literature about Strichartz inequalities on Euclidean space or manifolds, and it is beyond the scope of this introduction to review all of it. 
We instead mention few of the most relevant papers about  Strichartz estimates for the wave equation on the real hyperbolic spaces. 
On the real hyperbolic spaces $\mathbb{H}^n$,  Anker-Pierfelice \cite{AP}, 
Anker-Pierfelice-Vallarino \cite{APV}, Metcalfe-Taylor \cite{MT1,MT2} and Tataru \cite{T}
have showed better dispersive estimates and hence stronger results than in the Euclidean space. More precisely, they can obtain results with $(q,r)$ exterior of the range \eqref{adm}. Our first results will
generalize their results to any non-trapping asymptotically hyperbolic space, i.e. a  non-compact Riemannian manifold with variable curvature 
in which conjugate points can possibly appear, causing the failure of the usual dispersive estimate.\vspace{0.2cm}

\subsection{The setting} In this paper, we work on an $n$-dimensional complete non-compact Riemannian manifold $(M^\circ, g)$ where the metric $g$ is an asymptotically hyperbolic metric. 
This setting is the same as in Chen-Hassell
\cite{CH1,CH2}, Mazzeo \cite{M} and Mazzeo-Melrose\cite{MM}.  Let $x$ be a boundary defining function for the compactification $M$ of $M^\circ$.  We say a metric $g$ is conformally compact if $x^2g$ is
a Riemannian metric and extends smoothly up to the boundary $\partial M$. Mazzeo \cite{M} showed that its sectional curvature tends to $-|dx|^2_{x^2g}$ as $x\to0$; In particular, if the limit is such that $-|dx|^2_{x^2g}=-1$, we
say  that the conformally compact metric $g$ is asymptotically hyperbolic. More specifically, let $y=(y_1,\cdots,y_{n-1})$ be local coordinates on $Y=\partial
M$, and $(x,y)$ be the local coordinates on $M$ near $\partial M$; the metric $g$ in a
collar neighborhood $[0,\epsilon)_x\times \partial M$ takes
the form 
\begin{equation}\label{metric}
g=\frac{dx^2}{x^2}+\frac{h(x,y)}{x^2}=\frac{dx^2}{x^2}+\frac{\sum
h_{jk}(x,y)dy^jdy^k}{x^2},
\end{equation}
where $x\in C^{\infty}(M)$ is a boundary defining function for
$\partial M$ and $h$ is a smooth family of metrics on $Y=\partial M$.  In addition, if every geodesic
in $M$ reaches $\partial M$ both forwards and backwards, we say $M$ is
nontrapping.  The Poincar\'e disc $(\mathbb{B}^{n}, g)$ is a typical example of such manifold.  Indeed, considering the ball $\mathbb{B}^{n}=\{z\in\R^n: |z|<1\}$
endowed with the metric
\begin{equation}
g=\frac{4dz^2}{(1-|z|^2)^2},
\end{equation}
one can take $x=(1-|z|)(1+|z|)^{-1}$ as the boundary defining function and $\omega$ the coordinates on $\mathbb{S}^{n-1}$. Then the Poincar\'e metric 
takes the form $$g=\frac{dx^2}{x^2}+\frac{\frac14(1-x^2)^2 d\omega^2}{x^2},$$
where $d\omega^2$ is the standard metric on the sphere $\mathbb{S}^{n-1}$. 
Another typical example is the real hyperbolic space $\mathbb{H}^n$ which is a complete simply connected manifold of constant negative curvature $-1$.
Since the curvature is a negative constant, $\mathbb{H}^n$ is automatically non-trapping and has no conjugate points. 
\vspace{0.2cm}

\subsection{The main result about Strichartz estimate} 
Consider the wave equation associated to the Laplace-Beltrami operator $\Delta_g$ on the non-trapping asymptotically hyperbolic manifold $(M^\circ,g)$:
\begin{equation}\label{wave}
\begin{cases}
\partial_{t}^2u-\Delta_g u=F, \\ u(0)=u_0(z), ~\partial_tu(0)=u_1(z).
\end{cases}
\end{equation}
From Mazzeo-Melrose\cite{MM}, the continuous spectrum of $-\Delta_g$ is contained in $[\frac{(n-1)^2}4,+\infty)$, while the point spectrum is contained in $(0, \frac{(n-1)^2}4)$. 
When $-\Delta_g$ has no point spectrum, it is natural to consider a family of Klein-Gordon equations 
\begin{equation}\label{KG}
\begin{cases}
\partial_{t}^2u(t,z)-\Delta_g u(t,z)+m u(t,z)=F(t,z), \\ u(0)=u_0(z), ~\partial_tu(0)=u_1(z),
\end{cases}
\end{equation}
where the constant 
\begin{equation}
m\geq-\rho^2:=-(n-1)^2/4.
\end{equation}
In particular for $m=-\rho^2$, the equation is named the \emph{shifted} wave equation. In this paper, we focus on 
the shifted wave equation on any
non-trapping asymptotically hyperbolic manifold, motivated by the problem of small data global existence raised in \cite{SSW}.
Another motivation is to continue the study of dispersive equations on manifolds with variable curvature. As mentioned above, there possibly exist conjugate points in the variable curvature setting and they cause the failure of the usual dispersive estimates, but not of the Strichartz estimates. For example, on non-trapping asymptotically conic manifold whose curvature tends to zero as the boundary defining function $x\to0$, Hassell and the last author \cite{HZ, Zhang} established the global-in-time Strichartz estimates for Schr\"odinger and wave equations; which are the same as in Euclidean space. While on the non-trapping asymptotically hyperbolic manifold whose sectional curvature tends $-1$,  Chen \cite{Chen}
showed Strichartz estimates for the Schr\"odinger equation,  which are stronger than the Euclidean result. The crucial point in these papers is to use the microlocal  method to deal with the conjugate points of the manifold.
If the manifold has nonpositive curvature, e.g. the hyperbolic space $\mathbb{H}^n$  considered in \cite{AP, APV, MT1,MT2,T}, 
then there are no conjugate points. In the Euclidean space, the Strichartz estimates usually are proved by interpolating $L^2$-estimate and
a dispersive estimate. For Schr\"odinger equation, the dispersive estimate directly follows from the representation of the solution. While for the wave equation, the dispersive estimate requires a more complicated argument, which
typically involves Littlewood-Paley theory. However, in the hyperbolic setting, the usual Littlewood-Paley theory is missing, see Bouclet \cite{Bouclet}. To get around this, in the real hyperbolic space with constant sectional curvature $-1$, Metcalfe-Taylor \cite{MT1} 
made use of Sobolev spaces based on bmo-spaces, and interpolation results from \cite{Taylor}; Anker-Pierfelice \cite{AP} and
Anker-Pierfelice-Vallarino \cite{APV} used a good representation of fundamental solution of wave equation and complex interpolation argument. Before these works, Tataru \cite{T} obtained Strichartz estimates for $\mathbb H^n$ using complex interpolation.

For the variable curvature setting, we do not know such precise results. Of course, a standard replacement (which is very often sufficient) can be to use the Littlewood-Paley-Stein theory based on heat semi-groups;
see \cite{L, LOS}. In this case also, we could not overcome the technical issues. We take then a new approach. Our approach consists in splitting the solution space into low and high frequencies. We derive general Strichartz estimates, of independent interest, and use part of them (high frequencies) to obtain the global well-posedness for power-type nonlinearities. The argument crucially uses a microlocalized spectral measure estimate, which is a replacement for the argument involving restriction theorem (like Stein-Tomas theorem) for the Euclidean case.  \vspace{0.2cm}

Now we state our main result on the Strichartz estimate. Before doing so, we introduce some notation. 
Let $H=-\Delta_g-\rho^2$ and $\chi\in\mathcal{C}_c^\infty([0,\infty)$ such that $\chi(\lambda) = 1$ for $\lambda \leq 1$ and vanishes when $\lambda\geq2$.  
Define the norm of  $H^{a,b}_c$  by
\begin{equation}\label{space}
\|f\|_{H^{a,b}_c}=\|(1-\chi)(\sqrt{H})H^{\frac {a}2}f\|_{L^c}+\|\chi(\sqrt{H})H^{\frac {b}2}f\|_{L^c}.
\end{equation}
In the particular case $c=2$, we write briefly $H^{a,b}$. The space introduced here is an analogue of the usual Sobolev space but with separated regularity corresponding to high and low frequencies.
Next we define the sets related with the admissible conditions:
\begin{equation}\label{w-admissible}
\Lambda_w=\Big\{(q,r,\mu)\in [2,\infty]\times (2,\infty]\times\R: \frac{2}q\leq(n-1)(\frac12-\frac1r), \quad\mu>s_w\Big\},
\end{equation}
where 
\begin{equation}\label{sw}
s_w=n(\frac12-\frac1r)-\frac1q
\end{equation}
and 
\begin{equation}\label{e-admissible}
\Lambda_e=\Big\{(q,r,\mu)\in [2,\infty]\times (2,\infty]\times\R: \frac{2}q\geq(n-1)(\frac12-\frac1r), \quad\mu>s_e\Big\},
\end{equation}
where 
\begin{equation}\label{se}
s_e=\frac{n+1}2(\frac12-\frac1r).
\end{equation} 
We remark here that $\mu$ in the above sets is strictly greater than the optimal exponents $s_w$ and $s_e$. This fact will imply a loss of regularity for high frequencies in the Strichartz estimates. 

\begin{minipage}{0.5\textwidth}
\begin{center}
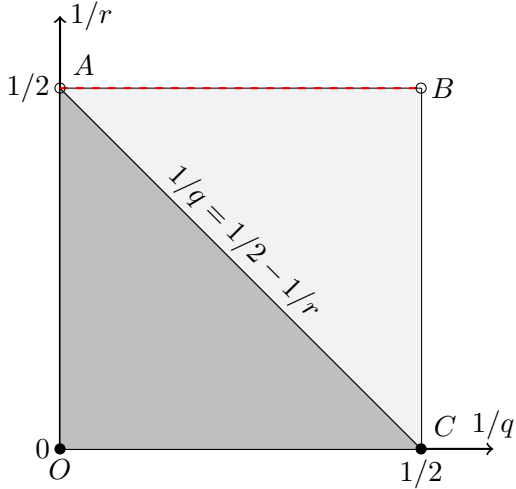

 \begin{tikzpicture}[scale=2.4]
  \draw[->,thick]  (0,0) -- (0,2.4);     
 \draw[->,thick]  (0,0) -- (2.4,0);  
 
\draw[thick] (0,0) -- (0,2)
                     (0,0) -- (2,0)
                     (0,2) -- (2,0)
                     (2,0) -- (2,2);

\draw (2,0) node[below] {$1/2$};
\draw (0,2) node[left] {$1/2$};
\draw (0,0) node[below, left] {$0$};
\draw (2.4,0) node[above] {$1/q$};
\draw (0,2.4) node[right] {$1/r$};

\filldraw[fill=gray!50](0,0)--(0,2)--(2,0); 
\filldraw[fill=gray!10](2,2)--(0,2)--(2,0); 

\draw[red, dashed, thick] (0, 2) -- (2,2); 

\draw (2,2) circle (0.8pt) node[right] {$B$};
\draw[fill=black](2,0) circle (0.8pt) node[above right=0.5mm of {(2,0)}] {$C$};
\draw  (0,2) circle (0.8pt) node[above right=0.5mm of {(0,2)}]{$A$};
\draw [fill=black] (0,0) circle (0.8pt) node[below] {$O$};

\node [rotate=-45] [above right= -0.5mm of {(0.5,1.5)}] {${1/q=1/2-1/r}$};
\end{tikzpicture}

\captionsetup{font=footnotesize}
\captionof{figure}{The range of $(q,r)$ when $n=3$. If $(q,r,\mu)\in\Lambda_w$, then $(q,r)$ is in
the triangle region ACO; While $(q,r,\mu)\in\Lambda_e$, then $(q,r)$ is in
the region ABC. }
\end{center}
\end{minipage}
\begin{minipage}{0.5\textwidth}
\begin{center}
 \begin{tikzpicture}[scale=2.4]
  \draw[->,thick]  (0,0) -- (0,2.4);     
 \draw[->,thick]  (0,0) -- (2.4,0);  
 
\draw[thick] (0,0) -- (0,2)
                     (0,0) -- (2,0)
                     (0,2) -- (2,1)
                     (2,0) -- (2,2);

\draw (2,0) node[below] {$1/2$};
\draw (0,2) node[left] {$1/2$};
\draw (0,0) node[below, left] {$0$};
\draw (2.4,0) node[above] {$1/q$};
\draw (0,2.4) node[right] {$1/r$};

\filldraw[fill=gray!50](0,0)--(0,2)--(2,0.5)--(2,0); 
\filldraw[fill=gray!10](2,2)--(0,2)--(2,0.5); 

\draw[red, dashed, thick] (0, 2) -- (2,2); 

\draw (2,2) circle (0.8pt) node[right] {$B$};
\draw [fill=black]  (2,0.5) circle (0.8pt) node[right] {$C$};
\draw [fill=black] (2,0) circle (0.8pt) node[above right=0.5mm of {(2,0)}] {$D$};
\draw (0,2) circle (0.8pt) node[above right=0.5mm of {(0,2)}]{$A$};
\draw [fill=black] (0,0) circle (0.8pt) node[below] {$O$};

\node [rotate=-35] [above right= 0.2mm of {(0.2,1.5)}] {${2/q=(n-1)(1/2-1/r)}$};
\end{tikzpicture}

\captionsetup{font=footnotesize}
\captionof{figure}{The range of $(q,r)$ when $n\geq4$. If $(q,r,\mu)\in\Lambda_w$, then $(q,r)$ is in
the region ACDO; While $(q,r,\mu)\in\Lambda_e$, then $(q,r)$ is in
the triangle region ABC. }
\end{center}
\end{minipage}
\vspace{0.2cm}

Our result about the homogeneous Strichartz estimate is the following.
\begin{theorem}[Homogeneous Strichartz estimate]\label{thm:hstri} Let  $(M^\circ,g)$ be any non-trapping asymptotically hyperbolic manifold of dimension
$n\geq2$ and let  $\Delta_g$ be the Laplace-Beltrami operator on $(M^\circ, g)$ and $\rho^2=(n-1)^2/4$. Assume $\Delta_g$ has no
pure point eigenvalue and has no resonance at the bottom of the continuous spectrum $\rho^2$. Suppose that $u$ is a
solution to the Cauchy problem
\begin{equation}\label{eq}
\begin{cases}
\partial_{t}^2u-\Delta_g u-\rho^2 u=0, \quad (t,z)\in I\times M^\circ; \\ u(0)=u_0(z),
~\partial_tu(0)=u_1(z),
\end{cases}
\end{equation}
for some initial data $u_0\in  H^{\mu,0}(M^\circ), u_1\in  H^{\mu-1,-\epsilon}(M^\circ)$ defined in \eqref{space}, and
the time interval $I\subseteq\R$, then
\begin{equation}\label{hstri}
\begin{split}
&\|u(t,z)\|_{L^q_t(I;L^r_z(M^\circ))}
\lesssim \|u_0\|_{H^{\mu,0}(M^\circ)}+\|u_1\|_{
H^{\mu-1,-\epsilon}(M^\circ)},
\end{split}
\end{equation}
where $(q,r,\mu)\in \Lambda_w\cup \Lambda_e $ defined in \eqref{w-admissible} and \eqref{e-admissible}, $0<\epsilon\ll1$. 
\end{theorem}

\begin{remark} 
The Strichartz estimate is global in time but with an arbitrary small loss in regularity of high frequency  which is a bit weaker than the estimates in \cite{T,AP, APV, MT1} on the hyperbolic space $\mathbb{H}^n$. 
The loss comes from our techniques since we lack  the (standard) Littlewood-Paley square function estimate due to the non-doubling property of the manifold (or a good representation of the fundamental solution  as in \cite{T,AP, APV}). 
\end{remark}

\begin{remark} 
Compared with \cite{T,AP, APV, MT1}, the general setting considered here may have conjugate points which can lead the usual dispersive estimate to fail. It is known that the sharp regularity Strichartz estimate in Euclidean space fails for admissible pairs (e.g. $q=2$ and $r=\infty$ when $n=3$), but we obtain the inequalities for the admissible pairs including  $q=2$.
\end{remark}

\begin{remark} 
We exclude the case $r=2$. At the special point $A$,  that is, $q=\infty, r=2$, the usual Strichartz estimate holds. For example, in the Euclidean space, the Strichartz estimate holds 
at $A$ if $\|u_0\|_{L^2}+\|u_1\|_{\dot H^{-1}}<\infty$; however one also can recover the estimate \eqref{hstri} at $A$ but with  $\epsilon=1$ by using Proposition \ref{energy} below.  
In this sense, the result gains some regularity in low frequency. 
\end{remark}

\begin{theorem}[Inhomogeneous Strichartz estimate] \label{thm:ihstri} Let $\Delta_g$ be as in Theorem \ref{thm:hstri}  and suppose that $u$ is a
solution to the Cauchy problem
\begin{equation}\label{eq}
\begin{cases}
\partial_{t}^2u-\Delta_g u-\rho^2 u=F(t,z), \quad (t,z)\in I\times M^\circ; \\ u(0)=0,
~\partial_tu(0)=0,
\end{cases}
\end{equation}
 and the time interval $I\subseteq\R$, then
 \begin{equation}\label{ihstri}
\begin{split}
&\|u(t,z)\|_{L^q_t(I;L^r_z(M^\circ))}
\lesssim \|F\|_{L^{\tilde{q}'}_t(I;H^{\mu+\tilde{\mu}-1,0}_{\tilde{r}'}(M^\circ))},
\end{split}
\end{equation}
where $(q,r,\mu), (\tilde{q},\tilde{r},\tilde{\mu})\in \Lambda_w\cup \Lambda_e $.
\end{theorem}

The proof of the inhomogeneous Strichartz estimate will be divided into two cases. The first case $q>\tilde{q}'$ is proved using the $TT^*$-method and the Christ-Kiselev lemma \cite{CK}. The second case when $q=\tilde{q}=2$ is more complicated to treat due to the failure of the Christ-Kiselev lemma and the usual dispersive estimate fails due to
the conjugate points. We overcome these difficulties  following the idea of Hassell and the last author \cite{HZ}. In this argument, we classify the microlocalized pseudo-differential operator via the wavefront set propagated along the bi-characteristic flow and parametrize the wavefront set off the diagonal case by a phase function with an unchanged sign. Finally we can show
some dispersive estimate in some special cases; see Proposition \ref{prop:Dispersive}  for details. Combining this with the $TT^*$-method again, we show the inhomogeneous Strichartz estimate when $q=\tilde{q}=2$.

\subsection{The small data global existence and Strauss conjecture}
We now apply the previous estimates to the nonlinear wave equation with small data. We introduce the class of nonlinearities: let $F_p\in C^1$ behaving like $\pm |u|^p$ or $\pm |u|^{p-1} u$, hence such that 
\begin{equation*}\label{eq-Fp}
|F_p(u)|+|u||F'_p(u)|\le C |u|^p,
\end{equation*}
for some constant $C>0$. Consider the family of nonlinear  equations 
\begin{equation}\label{NKG}
\begin{cases}
\partial_{t}^2u(t,z)-\Delta_g u(t,z)+m u(t,z)=F_p(u), \\ u(0)=u_0(z), ~\partial_tu(0)=u_1(z),
\end{cases}
\end{equation}
where the constant satisfies
\begin{equation}
m\geq-\rho^2:=-(n-1)^2/4.
\end{equation}

The problem under consideration belongs to the realm of the dichotomy between global existence vs blow-up for the nonlinear equation \eqref{NKG} with $m=0$ as investigated for the first time by F. John in \cite{John79} on the Euclidean space. 
John determined the critical power to be $p_S=1+\sqrt{2}$ for the problem when $n=3$ by
proving  global existence results for $p>1+\sqrt{2}$ and
blow-up results for $p<1+\sqrt{2}$. Later, Strauss \cite{Strauss81} 
conjectured that the critical power $p_c(n)=p_S(n)$ (above which global existence for small data holds) for other dimensions $n\ge 2$ should be the positive root of the quadratic equation $$(n - 1)p^2
- (n + 1)p - 2 = 0.$$
See \cite{Wang18} and the references therein for a complete account on the state of the art. On the real hyperbolic space $\mathbb{H}^n$,  
Metcalfe-Taylor \cite{MT1} gave a proof of small data global existence for \eqref{NKG} with $m=0$ and $p\geq5/3$ for dimension $n=3$ and then Anker-Pierfelice \cite{AP} proved global existence for the problem \eqref{NKG} with $m>-\rho^2$ and $p \in (1,1+\frac{4}{n-1}]$
where $n\geq2$.  Metcalfe-Taylor \cite{MT2} gave an alternative proof for $n=3$. Notice that  the spectrum of the Laplacian on $\mathbb{H}^n$
is contained in $[\rho^2,\infty)$; these results  are more like a nonlinear Klein-Gordon equation instead of a nonlinear wave equation.
  For the limit case $m=-\rho^2$, i.e., the shifted wave equation, Fontaine \cite{fontaine} was the first one to provide small data global existence for $n=2,3$ and $p \geq 2$. Anker-Pierfelice-Vallarino \cite{APV}
proved wider couples of Strichartz estimates and a stronger local well-posedness result for the nonlinear shifted wave equation.
The Strichartz estimate established in \cite{APV} could be applied to show small data global existence for any $p\in (1, 1+4/(n-1)]$, even though such results 
have not been proved explicitly in \cite{APV}; hence it illustrates the critical power of global existence holds for the shifted wave equation with small data $p_c(n)=1$.
This result on $\mathbb{H}^n$ is explicitly stated and proved by the first three authors \cite{SSW}. Tataru \cite{T} actually proved dispersive estimates, which are strong enough to ensure global results, as pointed out in \cite{SSW}. On Damek-Ricci spaces (which contain Riemannian symmetric spaces of rank one), Anker-Pierfelice-Vallarino \cite{APV2} prove also global results.
In \cite{SSW}, the authors also showed 
small data global existence result for \eqref{NKG} on a manifold with variable curvature under the assumption that $\mathrm{Spec}(-\Delta_g+m)\subset (c,+\infty)$ with $c>0$. The final remark of \cite{SSW} raised a question about the small data global existence for \eqref{NKG} with $p>1$ and $m=-\kappa\rho^2$ on a manifold with variable negative curvature 
with sectional curvatures $K\in [-\tilde{\kappa}, -\kappa]$ for some $\tilde{\kappa}\geq \kappa>0$. Our second result partially answers this problem. More precisely, we prove 

\begin{theorem}\label{GWP}
Let  $(M^\circ,g)$ be a non-trapping asymptotically hyperbolic manifold of dimension
$n\geq2$ and let  $\Delta_g$ be the Laplace-Beltrami operator on $(M^\circ, g)$ as in Theorem \ref{thm:hstri}. Let $\rho^2=(n-1)^2/4$ and $p \in (1,1+\frac{4}{n-1})$.   Then there exists a constant $\nu_1>0$ such that  the Cauchy problem
\begin{equation}\label{eqNLW}
\begin{cases}
\partial_{t}^2u-\Delta_g u-\rho^2 u=F_p(u), \quad (t,z)\in I\times M^\circ; \\ u(0)=\nu u_0(z),
~\partial_tu(0)=\nu u_1(z),
\end{cases}
\end{equation}
 has  global solution where $\nu \in (0, \nu_1]$ and $u_0\in  H^{\mu,0}(M^\circ), u_1\in  H^{\mu-1,-\epsilon}(M^\circ)$ defined in \eqref{space} for $\epsilon$ very small and $\mu>\frac{n+1}2(\frac12-\frac1{p+1})$.
\end{theorem}

\begin{remark}
The assumption on the regularity of the initial data is not sharp. The usual investigations for small data global existence requires more care, see Wang \cite{Wang18}.

\end{remark}

Notice here that we do not reach the endpoint $p=1+\frac{4}{n-1}$. The reason is that there is a loss of derivatives in the inhomogeneous Strichartz estimates and so it is impossible to close the iteration in this latter case. In this case, one has to use another method, based on the strategy described in \cite{BSS}. We postpone this issue to a later work since the techniques are very different.\vspace{0.2cm}

\emph{ Notation.} We use $A\lesssim B$ to denote
$A\leq CB$ for some large constant C which may vary from line to
line and depend on various parameters, and similarly we use $A\ll B$
to denote $A\leq C^{-1} B$. We employ $A\sim B$ when $A\lesssim
B\lesssim A$. If the constant $C$ depends on a special parameter
other than the above, we shall denote it explicitly by subscripts.
For instance, $C_\epsilon$ should be understood as a positive
constant not only depending on $p, q, n$, and $M$, but also on
$\epsilon$. Throughout this paper, pairs of conjugate indices are
written as $p, p'$, where $\frac{1}p+\frac1{p'}=1$ with $1\leq
p\leq\infty$. \vspace{0.2cm}

\emph{ Organization of this paper. } Our paper is organized as follows. We recall the properties of the microlocalized spectral measure in Section 2. In Section 3, we define the microlocalized propagator and prove the energy estimate and the microlocalized dispersive estimate. We conclude this section by showing the microlocalized Strichartz estimate. We prove Theorem \ref{thm:hstri} in Section 4 and Theorem \ref{thm:ihstri} in Section 5. Finally, we prove the global existence of Theorem \ref{GWP} in the section 6.

 \vspace{0.2cm}

{\bf Acknowledgments:}\quad  We thank Jean-Marc Bouclet and Andrew Hassell for helpful discussions. Y. S. was partially supported by the Simons foundation. 
C. D. S. was supported by the NSF and the Simons foundation.
C. Wang was supported in part by NSFC 11971428 and National Support Program for Young Top-Notch Talents.
 J. Zhang was supported by  NSFC Grants (11771041, 11831004) and H2020-MSCA-IF-2017(790623). \vspace{0.2cm}

\section{The spectral measure }

In this section, we briefly review the key elements of the
microlocalized spectral measure, which was constructed and proved by Chen-Hassell
\cite [Theorem 1.3]{CH1,CH2}. This is an analogue of a result of  Hassell and the last author \cite[Proposition 1.5]{HZ} for the non-trapping asymptotically conic manifold. The property not only gives the decay of spectral measure 
in frequency but also captures the oscillatory behaviour of the 
spectral measure.

\begin{proposition}
\label{prop:localized spectral measure} Let $(M^\circ,g)$ and
$H=-\Delta_g-(n-1)^2/4$ be as in Theorem \ref{thm:hstri}. 
Then for low energy, i.e. $\lambda\leq 2$, the Schwartz kernel of the spectral measure $dE_{\sqrt{H}}(\lambda; z,z')$ satisfies 
\begin{equation}\label{lbeans}
dE_{\sqrt{H}}(\lambda; z,z')=\lambda\left((\rho_L\rho_R)^{\frac {n-1}2+i\lambda}a(\lambda;z,z')-(\rho_L\rho_R)^{\frac {n-1}2-i\lambda}a(-\lambda;z,z')\right),
\end{equation}
where $a\in \mathcal{C}^\infty([-2,2]_{\lambda}\times M_0^2)$
and $\rho_L$ and $\rho_R$ are respectively the boundary defining functions for the left and right boundary in the double space $M_0^2$. Furthermore, there holds 
\begin{equation}\label{lbeans'}
|dE_{\sqrt{H}}(\lambda; z,z')|\leq C \lambda^2(1+d(z,z')) e^{-(n-1)d(z,z')/2}.
\end{equation}

For the high energy, i.e. $\lambda\geq1/2$, there exists a finite pseudo-differential operator partition of the identity operator 
\begin{equation}\label{partition}
\begin{split}
\mathrm{Id}=\sum_{k=0}^{N}Q_k(\lambda),
\end{split}
\end{equation}
where the $Q_k$ are
uniformly bounded as operators on $L^2$ and $N$ is independent of $\lambda$, such that
\begin{equation}\label{bean}
(Q_k(\lambda)dE_{\sqrt{H}}(\lambda)Q_k^*(\lambda))(z,z')=\lambda^{n-1}\sum_{\pm}e^{\pm
i\lambda d(z,z')}a_{\pm}(\lambda;z,z')+b(\lambda;z,z'),
\end{equation}
where $d(\cdot, \cdot)$ is the Riemannian distance on $M^\circ$, and for any $\alpha$, there exists a constant $C_\alpha$ such that
the $a_\pm$ satisfies
\begin{equation}\label{beans}
|\partial_\lambda^\alpha a_\pm(\lambda;z,z')|\leq \begin{cases}
C_\alpha
\lambda^{-\alpha}(1+\lambda d(z,z'))^{-\frac{n-1}2},\quad d(z,z')\leq 1,\\
C_{\alpha}\lambda^{-\frac {n-1}2-\alpha}e^{-(n-1)d(z,z')/2},\quad d(z,z')\geq1,
\end{cases}
\end{equation}
and $b$ satisfies 
\begin{equation}\label{beans'}
| \partial_\lambda^\alpha b(\lambda;z,z')|\leq C_\alpha \lambda^{-K-\alpha} e^{-(n-1)d(z,z')/2}, \quad \forall \alpha,\, K>0.
\end{equation}
Moreover, if $(M^\circ,g)$ is in addition simply connected with nonpositive sectional curvatures, then the estimates above are true for spectral measure without
microlocalization, that is,  in this case we can take $\{Q_k(\lambda)\}$ to be the trivial partition of unity.
\end{proposition}

\begin{remark}
For example, a Cartan-Hadamard manifold is a simply connected manifold with nonpositive sectional curvatures, hence we have the estimates above 
without microlocalization. The nonpositive sectional curvatures imply that the manifold is non-trapping and has no conjugate points.
\end{remark}

Next we show an inequality for an integral operator which is similar to a result of Anker-Pierfelice-Vallarino \cite{APV} on $\mathbb{H}^n$.  
This is close to a non-Euclidean feature of hyperbolic space related to the Kunze-Stein phenomenon
\cite{KS}. 

\begin{lemma}\label{lem:qq'}
Let $M^\circ$ be the manifold as in Theorem \ref{thm:hstri} and let the kernel $K$ satisfy the pointwise bound
\begin{equation}\label{ker}
|K(z,z')|\leq e^{-\rho_\delta d(z,z')}, \quad \rho_\delta=\rho-\delta=(n-1)/2-\delta,
\end{equation}
then for any $q\in (2,\infty]$, there exists a constant $C$ and $0<\delta_0(q):=(n-1)(\frac12-\frac1q)$ such that
\begin{equation}\label{est:qq'}
\big\|\int_{M^\circ} K(z,z') f(z') dg(z')\big\|_{L^q(M^\circ)}\leq C \|f\|_{L^{q'}(M^\circ)}
\end{equation}
holds for all $0<\delta< \delta_0$.

\end{lemma}

\begin{proof} The proof is a variant of the argument in \cite[Section 4.2]{CH2} where the estimates of the spectral measure
are established. We show that there exists a constant $C$ such that
\begin{equation}\label{vest:qq'}
\begin{split}
\Big|\int_{M^\circ\times M^\circ} K(z,z') f(z')  h(z)dg(z')dg(z)\Big|\leq C \|f\|_{L^{q'}}\|h\|_{L^{q'}}.
\end{split}
\end{equation}
We split the left hand side into several pieces, restricting the kernel to different regions. 
Recall that $M$ is the compactification of $M^\circ$ and $M^2_0$ is the blow up space. 
Let $O$ be a neighbourhood of the front face FF in $M^2_0$.  We write 
$$K(z,z')=K(z,z') \chi_{O} +K(z,z') \chi_{M^2_0\setminus O},$$
where $\chi$ is the usual bump function. We first consider \eqref{vest:qq'} with the kernel $K(z,z') \chi_{M^2_0\setminus O}$.
Since the other cases are similar, we only prove  \eqref{vest:qq'} when both $z,z'$ are near the boundary $\{x=0\}$ of $M$ where $x$ is the boundary defining function. 
Away from the front face, the distance $d(z,z')$ is comparable to $-\log(xx')$.
Since $q>2$ and  $0<\delta\ll 1$, we obtain
\begin{equation}
\begin{split}
&\Big|\int_{\{x, x' \leq \eta \}\cap M^2_0\setminus O} K(z,z') f(z')  h(z)dg(z')dg(z)\Big|
\\& \lesssim \int_{x, x' \leq \eta} (xx')^{q\rho_\delta}\frac{dx'}{x'^{n}}\frac{dx}{x^{n}} \|f\|_{L^{q'}}\|h\|_{L^{q'}}\leq C_{\eta} \|f\|_{L^{q'}}\|h\|_{L^{q'}}.
\end{split}
\end{equation}

Now consider the kernel $K(z,z') \chi_{O}$ near the front face. Further decompose the set $O$ into subsets $O_i\subset M^2_0$, 
$$O_i=\{(x,x',y,y'): x, x'\leq \eta; \quad d_h(y,y_i), d_h(y',y_i)\leq \eta\}$$
for some $y_i\in \partial M$ where the distance $d_h$ is measured by the metric $h(0)$ on $\partial M$. Use the local coordinates  $(x,y)$
on $M$ which is near $(0,y_i)\in \partial M$ to define a map $\phi_i$ such that
\begin{equation}
\phi_i:  U_i\mapsto U_i'\subset \mathbb{H}^{n},
\end{equation}
where $U_i=\{(x,y)\in M :x\leq \eta, d_h(y,y_i)\leq \eta\}$ and $U'_i$ is a neighbourhood of the origin $(0,0)$ (using the upper half-space model) in the real hyperbolic space $\mathbb{H}^{n}$.
The map $\phi_i$ induces a diffeomorphism $\Phi_i$
\begin{equation}
\Phi_i:  O_i\mapsto O_i',
\end{equation}
where $O'_i$ is a subset of  $(\mathbb{B}^{n})_0^2$, the double space for $\mathbb{H}^n$. Let $r$ be the geodesic distance on $\mathbb{H}^{n}$,
then the kernel satisfies
\begin{equation}\label{ker'}
|\phi_i \circ K(z,z')\chi_{O_i} \circ \phi_i^{-1}|\leq C e^{-\rho_\delta r}.
\end{equation}

We need the following lemma proved in \cite[Lemma 5.1]{APV}:
\begin{lemma} Let $q\geq 2$,
\begin{equation}
\|f\ast \kappa\|_{L^q(\mathbb{H}^{n})}\leq C_q \|f\|_{L^{q'}(\mathbb{H}^n)}\left(\int_0^\infty (\sinh r)^{n-1}(1+r)e^{-(n-1)r/2}|\kappa(r)|^{q/2}dr\right)^{2/q}.
\end{equation}
\end{lemma}
Using this lemma with $\kappa(r)=e^{-\rho_\delta r}$ and the fact 
\begin{equation}
\begin{split}
&\int_0^\infty (\sinh r)^{n-1}(1+r)e^{-(n-1)r/2}|\kappa(r)|^{q/2}dr
\\&\leq \int_0^\infty (1+r) e^{-\frac{n-1}2(\frac q2-1)r}e^{\frac {q\delta}2r}dr<\infty,\quad 0<\delta<\delta_0,
\end{split}
\end{equation} we obtain that the integral operator with kernel \eqref{ker'} is bounded from $L^{q'}(\mathbb{H}^n)$ into $L^{q}(\mathbb{H}^n)$.
Therefore it shows the integral operator with the kernel $K(z,z') \chi_{O}$ is bounded from $L^{q'}(U_i)$ to $L^{q}(U_i)$ since $\phi_i$ are bounded and invertible maps from 
$L^{q'}(U_i)$ to $L^{q'}(U'_i)$.

\end{proof}

\section{Dispersive estimate and microlocalized Strichartz estimate}

In this section, we define the microlocalized wave propagator and prove the  microlocalized $L^2$-estimates and the dispersive estimates.
As a final conclusion of this section, we prove  microlocalized Strichartz estimates.

\subsection{Microlocalized wave propagator and $L^2$-estimates}
We first define the microlocalized wave propagator. Denote $U(t)=e^{it\sqrt{H}}$. For any $\sigma\in\R$, we define 
\begin{equation}
\leftidx{^{\sigma}}U(t)=e^{it\sqrt{H}}H^{-\frac\sigma 2}.
\end{equation}
In the following application, in particular, we are interested in the cases $\sigma=0$, $\sigma=1/2$ and $\sigma=1$.
Choose $\chi \in C_c^\infty(\R)$,
such that $\chi(\lambda) = 1$ for $\lambda \leq 1$ and vanishes when $\lambda\geq2$. Then we define
\begin{equation}\label{Ulh}
\leftidx{^{\sigma}}U^{\mathrm{low}}(t) = \int_0^\infty e^{it\lambda} \chi(\lambda)\lambda^{-\sigma}
dE_{\sqrt{H}}(\lambda) , \quad \leftidx{^{\sigma}}U^{\mathrm{high}}(t) =
\int_0^\infty e^{it\lambda} (1 - \chi)(\lambda)
\lambda^{-\sigma}dE_{\sqrt{H}}(\lambda).
\end{equation}
Let $\varphi\in C_c^\infty([1/2, 2])$ and take values in
$[0,1]$ such that $1=\sum_{j\in\Z}\varphi(2^{-j}\lambda)$ for any $\lambda \neq 0$.
We define
\begin{equation}\label{Ul-mic}
\begin{split}
\leftidx{^{\sigma}}U^{\mathrm{low}}_j(t) &= \int_0^\infty e^{it\lambda}
\varphi(2^{-j}\lambda)\chi(\lambda) \lambda^{-\sigma}dE_{\sqrt{H}}(\lambda),\\
\leftidx{^{\sigma}}U^{\mathrm{high}}_j(t) &= \int_0^\infty e^{it\lambda}
\varphi(2^{-j}\lambda)(1 - \chi)(\lambda)
\lambda^{-\sigma}dE_{\sqrt{H}}(\lambda).
\end{split}
\end{equation}
For the high-energy operator partition of
identity operator $Q_k(\lambda)$ in Proposition \ref{prop:localized spectral measure}, we further define
\begin{equation}\label{Uh-mic}
\begin{gathered}
\leftidx{^{\sigma}}U^{\mathrm{high}}_{j,k}(t) = \int_0^\infty e^{it\lambda} \varphi(2^{-j}\lambda)(1 -
\chi)(\lambda)\lambda^{-\sigma}Q_{k}(\lambda)
dE_{\sqrt{H}}(\lambda),~ 0\leq k\leq N.
\end{gathered}
\end{equation}
The above definition of the operator is well-defined. Indeed, we have
\begin{proposition}[$L^2$-estimates]\label{energy} Let $\leftidx{^{\sigma}}U^{\mathrm{low}}_j(t) $ and $\leftidx{^{\sigma}}U^{\mathrm{high}}_{j,k}(t) $ be defined in \eqref{Ul-mic} and  \eqref{Uh-mic}.
Then there exists a constant $C$ independent of $t, j, k$ such that
\begin{equation}\label{L2-est}
\begin{split}
\|\leftidx{^{\sigma}}U^{\mathrm{low}}_j(t)\|_{L^2\rightarrow L^2}\leq C 2^{-\sigma j},\quad
\|\leftidx{^{\sigma}}U^{\mathrm{high}}_{j,k}(t)\|_{L^2\rightarrow L^2}\leq C2^{-\sigma j}
\end{split}
\end{equation}
 for all $k\geq 0,
j\in\Z$.
\end{proposition}

\begin{remark}
The estimate of $\leftidx{^{\sigma}}U^{\mathrm{low}}_j(t)$ will not be used in the following proofs.
In the following argument, we only need estimates of $\leftidx{^{\sigma}}U^{\mathrm{high}}_{j,k}(t)$ for the interpolation argument.
\end{remark}

\begin{proof} The proof essentially follow the argument in \cite{HZ, Zhang} in which Hassell and the last author considered the cases of asymptotically conic manifolds. One also can find a modified version in \cite{Chen} on the asymptotically hyperbolic setting. We here outline the proof for the convenience of the reader.  \vspace{0.2cm}

We first show that the above definition of the operator is well-defined.  To this end, it suffices
to show the above integrals in the definitions are well-defined over any compact
dyadic interval in $(0, \infty)$. Let $A(\lambda) = e^{it\lambda}
\chi(\lambda) \varphi(2^{-j}\lambda)\lambda^{-\sigma}$ or $
A(\lambda) =e^{it\lambda}\varphi(2^{-j}\lambda)(1-\chi)(\lambda)\lambda^{-\sigma}
Q_{k}(\lambda)$. Then $A(\lambda)$ is compactly supported in $[a,b]$ with $a=2^{j-1}$ and $b=2^{j+1}$
and $\mathcal{C}^1$ in $\lambda\in (0,\infty)$. After integrating by parts, we see that the
integral 
$$
\int_a^b A(\lambda) dE_{\sqrt{H}}(\lambda)
$$
is given by
\begin{equation}\label{mean}
E_{\sqrt{H}}(b) A(b) - E_{\sqrt{H}}(a)
A(a) - \int_a^b \frac{d}{d\lambda} A(\lambda)
E_{\sqrt{H}}(\lambda) \, d\lambda.
\end{equation}
From the construction of the pseudo-differential operator $Q_k(\lambda)$ in \cite[Section 6.1]{CH2}, 
similarly as \cite[Corollary 3.3]{HZ}, we can show $Q_k(\lambda)$ and each operator $\lambda
\partial_\lambda Q_k(\lambda)$ is bounded on $L^2(M^\circ)$
uniformly in $\lambda$. Then this means that the integrals are well-defined over any dyadic compact
interval in $(0, \infty)$, hence the operators $U^{\mathrm{low}}_j(t)$ and $U^{\mathrm{high}}_{j,k}(t)$ are
well-defined. 

Next we show these operators are bounded on $L^2$. We
only consider $\leftidx{^{\sigma}}U^{\mathrm{high}}_{j,k}(t)$ since the other is handled in the same way. We have by  \cite[Lemma 5.3]{HZ},
\begin{equation}\begin{gathered}
\leftidx{^{\sigma}}U^{\mathrm{high}}_{j,k}(t) \leftidx{^{\sigma}}U^{\mathrm{high}}_{j,k}(t)^* = \int   (1-\chi)^2(\lambda) \varphi\big(
\frac{\lambda}{2^j} \big) \varphi\big( \frac{\lambda}{2^j} \big)\lambda^{-2\sigma}
Q_k(\lambda) dE_{\sqrt{H}}(\lambda) Q_k(\lambda)^* \\
= -\int \frac{d}{d\lambda} \Big( (1-\chi)^2(\lambda) \varphi\big(
\frac{\lambda}{2^j} \big) \varphi\big( \frac{\lambda}{2^j} \big)
Q_k(\lambda) \lambda^{-2\sigma}\Big) E_{\sqrt{H}}(\lambda) Q_k(\lambda)^*  \\
- \int  (1-\chi)^2(\lambda) \varphi\big( \frac{\lambda}{2^j} \big)
\varphi\big( \frac{\lambda}{2^j} \big)\lambda^{-2\sigma} Q_k(\lambda)
E_{\sqrt{H}}(\lambda) \frac{d}{d\lambda}
Q_k(\lambda)^*.
\end{gathered}\label{Uijk}\end{equation}
On one hand, we note that this is independent of $t$ and also recall that $Q_k(\lambda)$ and each operator $\lambda
\partial_\lambda Q_k(\lambda)$ is bounded on $L^2(M^\circ)$
uniformly in $\lambda$. On the other hand,  the
integrand is a bounded operator on $L^2$, with an operator bound of
the form $C\lambda^{-1-2\sigma}$ where $C$ is uniform and by the support property of $\varphi$, then the
$L^2$-operator norm of the integral
is therefore uniformly bounded by $2^{-2j\sigma}$, as we are integrating over a dyadic
interval in $\lambda$ and the proposition is proved. 

\end{proof}

\subsection{Dispersive estimates} In this subsection, we prove the microlocalized dispersive estimates which are the
key estimates to derive the Strichartz estimates.

\begin{proposition}\label{prop:m-disper}
Let $\leftidx{^{\sigma}}U^{\mathrm{low}}_j(t) $ and $\leftidx{^{\sigma}}U^{\mathrm{high}}_{j,k}(t) $ be defined in \eqref{Ul-mic} and  \eqref{Uh-mic}.
Let $\rho=(n-1)/2$.  Then there exists constants $C$ independent of $t, j, k$ for all
$ j\in\Z$ such that

$\bullet$ For $j\geq 0$, $\sigma\geq 0$ and $|t-\tau|\leq 2$
\begin{equation}\label{Dispersiveh}
\begin{split}
\|\leftidx{^{\sigma}}U^{\mathrm{high}}_{j,k}(t)&(\leftidx{^{\sigma}}U^{\mathrm{high}}_{j,k}(\tau))^*\|_{L^1\to L^\infty}
\\&\leq C
2^{j[(n+1)/2-2\sigma]}(2^{-j}+|t-\tau|)^{-(n-1)/2};
\end{split}
\end{equation}

$\bullet$  For $j\geq 0$, $\sigma\geq0$ and $|t-\tau|\geq 2$
\begin{equation}\label{Dispersiveh'}
\begin{split}
\|\leftidx{^{\sigma}}U^{\mathrm{high}}_{j,k}(t)&(\leftidx{^{\sigma}}U^{\mathrm{high}}_{j,k}(\tau))^*\|_{L^1\to L^\infty}
\\&\leq C
2^{j[(n+1)/2-2\sigma]}|t-\tau|^{-K},\quad \forall K\geq0;
\end{split}
\end{equation}

$\bullet$  For $j\leq 0$, $0\leq\sigma<3/2$ and $0\leq \epsilon\ll \min\{1, 3-2\sigma\}$
\begin{equation}\label{Dispersivel}
\|\leftidx{^{\sigma}}U^{\mathrm{low}}_{j}(t)(\leftidx{^{\sigma}}U^{\mathrm{low}}_{j}(\tau))^*\|_{L^1\to L^\infty}\leq C
2^{\mp\epsilon j}(1+|t-\tau|)^{2\sigma-3\mp\epsilon}.
\end{equation}

\end{proposition}

\begin{proof} As before, we have by  \cite[Lemma 5.3]{HZ}, 
\begin{equation}\begin{gathered}\label{h-express}
\leftidx{^{\sigma}}U^{\mathrm{high}}_{j,k}(t) (\leftidx{^{\sigma}}U^{\mathrm{high}}_{j,k}(\tau))^* = \int_0^\infty e^{i(t-\tau)\lambda}(1-\chi)^2(\lambda) \varphi^2\big(
\frac{\lambda}{2^j} \big) \lambda^{-2\sigma}
Q_k(\lambda) dE_{\sqrt{H}}(\lambda) Q_k(\lambda)^* 
\end{gathered}\end{equation}
and 
\begin{equation}\begin{gathered}\label{l-express}
\leftidx{^{\sigma}}U^{\mathrm{low}}_{j}(t) (\leftidx{^{\sigma}}U^{\mathrm{low}}_{j}(\tau))^* = \int_0^\infty e^{i(t-\tau)\lambda}\chi^2(\lambda) \varphi^2\big(
\frac{\lambda}{2^j} \big) \lambda^{-2\sigma} dE_{\sqrt{H}}(\lambda).
\end{gathered}\end{equation}
Let $\phi(\lambda)=\varphi^2(\lambda)$. Then the proposition is a consequence of the following lemma about the microlocalized dispersive estimates.

\begin{lemma}[Microlocalized dispersive estimates]\label{lem:dispersive}
Let $Q(\lambda)$ be the operator
$Q_k$ constructed as in Proposition \ref{prop:localized spectral measure}
and suppose $\phi\in C_c^\infty([1/2, 2])$ and takes value in
$[0,1]$. Let $\rho=(n-1)/2$ and $0<\delta\ll1$. Then, for $j\geq 0$ and any $\sigma\geq0$, there exist positive constant $C$ independent of $j$ and points
$z,z'\in M^\circ$ such that

$\bullet$ when $|t|\leq 2$

\begin{equation}\label{dispersiveh}
\begin{split}
\Big|\int_0^\infty e^{it\lambda}\phi(2^{-j}\lambda) (1-\chi)^2(\lambda)&\lambda^{-2\sigma}\big(Q(\lambda)
E'_{\sqrt{H}}(\lambda)Q^*(\lambda)\big)(z,z')
d\lambda\Big|\\&\leq C 2^{j[(n+1)/2-2\sigma]}(2^{-j}+|t|)^{-(n-1)/2}e^{-(\rho-\delta)d(z,z')},
\end{split}
\end{equation}

$\bullet$ when $|t|\geq 2$
\begin{equation}\label{dispersiveh'}
\begin{split}
\Big|\int_0^\infty e^{it\lambda}\phi(2^{-j}\lambda) (1-\chi)^2(\lambda) &\lambda^{-2\sigma}\big(Q(\lambda)
E'_{\sqrt{H}}(\lambda)Q^*(\lambda)\big)(z,z')
d\lambda\Big|\\&\leq C 2^{j[(n+1)/2-2\sigma]}|t|^{-K}e^{-(\rho-\delta)d(z,z')},\quad \forall K\geq 0;
\end{split}
\end{equation}
and for $j\leq 0$, there exists constant $C$ independent of $j$ and points
$z,z'\in M^\circ$ such that
\begin{equation}\label{dispersivel'}
\begin{split}
\Big|\int_0^\infty & e^{it\lambda}\phi(2^{-j}\lambda)\chi^2(\lambda)\lambda^{-2\sigma} 
E'_{\sqrt{H}}(\lambda; z,z')
d\lambda\Big|\\&\leq C 2^{\mp\epsilon j}(1+|t|)^{2\sigma-3\mp\epsilon}e^{-(\rho-\delta)d(z,z')},\, 0\leq \sigma<3/2, \, 0\leq \epsilon\ll \min\{1, 3-2\sigma\}.
\end{split}
\end{equation}
\end{lemma}

Note that $dE_{\sqrt{H}}(\lambda)=E'_{\sqrt{H}}(\lambda) d\lambda$, thus we have proved the result in Proposition \ref{prop:m-disper} once we prove the lemma.

\end{proof}

\begin{remark} In the proof of Proposition \ref{prop:m-disper}, the factor $e^{-(\rho-\delta)d(z,z')}$ is used as a bounded constant.
This is enough to obtain the high frequency estimate \eqref{stri} in Proposition \ref{m-strichartz} below. 
However, the factor $e^{-(\rho-\delta)d(z,z')}$ is needed to obtain the low frequency estimates  \eqref{stri'} and \eqref{stri''}.
\end{remark}

\begin{proof}[The proof of Lemma \ref{lem:dispersive}] We shall rely on Proposition \ref{prop:localized spectral measure}. 
We first prove \eqref{dispersiveh} and \eqref{dispersiveh'} which are for the high frequencies. Using Proposition \ref{prop:localized spectral measure}, it suffices to estimate 
\begin{equation}\label{a}
\int_0^\infty e^{it\lambda}\phi(2^{-j}\lambda)\lambda^{n-1-2\sigma}e^{\pm i\lambda d(z,z')}\tilde{a}_{\pm}(\lambda; z,z')
d\lambda
\end{equation}
and
\begin{equation}\label{b}
\int_0^\infty e^{it\lambda}\phi(2^{-j}\lambda)\lambda^{-2\sigma}\tilde{b}(\lambda; z,z')
d\lambda,
\end{equation}
where $\tilde{a}_\pm=(1-\chi)^2(\lambda)a_{\pm}(\lambda; z,z')$ and $\tilde{b}=(1-\chi)^2(\lambda)b(\lambda; z,z')$ with $a_\pm$ and $b$ satisfying \eqref{beans} and \eqref{beans'}.
It is easy to verify that $\tilde{a}_\pm$ and $\tilde{b}$ have the same property as $a_\pm$ and $b$ respectively, that is, $\tilde{a}_\pm$ satisfies \eqref{beans}  and $\tilde{b}$ satisfies \eqref{beans'}. Hence we briefly relabel $\tilde{a}_\pm$ to $a_\pm$ and $\tilde{b}$ to $b$ without confusion from now on.\vspace{0.2cm}

For any $K>0$, we have by  \eqref{beans'} in Proposition \ref{prop:localized spectral measure}
\begin{equation*}
\begin{split}
\Big|\int_0^\infty e^{it\lambda}\phi(2^{-j}\lambda)\lambda^{-2\sigma}b(\lambda; z,z')
d\lambda\Big| &\leq \int_0^\infty \phi(2^{-j}\lambda) \lambda^{-K-2\sigma} d\lambda \,e^{-(n-1)d(z,z')/2}
\\& \leq 2^{j(1-K-2\sigma)}e^{-(n-1)d(z,z')/2}.
\end{split}
\end{equation*}
 We use \eqref{beans'} and  $N$ integrations by parts to obtain
\begin{equation*}
\begin{split}
&\Big|\int_0^\infty e^{it\lambda}\phi(2^{-j}\lambda)\lambda^{-2\sigma}b(\lambda; z,z')
d\lambda\Big|\\&\leq \Big|\int_0^\infty \big(\frac1{
it}\frac\partial{\partial\lambda}\big)^{N}\big(e^{it\lambda}\big)
\phi(2^{-j}\lambda)\lambda^{-2\sigma}b(\lambda; z,z')
d\lambda\Big|\\& \leq
C_N|t|^{-N}\int_{2^{j-1}}^{2^{j+1}}\lambda^{-K-N-2\sigma}d\lambda \, e^{-(n-1)d(z,z')/2}\leq
C_N|t|^{-N}2^{j(1-K-N-2\sigma)}e^{-(n-1)d(z,z')/2}.
\end{split}
\end{equation*}
Note that $j\geq 0$, therefore we obtain
\begin{equation}\label{dispersive1}
\begin{split}
&\Big|\int_0^\infty e^{it\lambda} \phi(2^{-j}\lambda)\lambda^{-2\sigma}b(\lambda;z,z')
d\lambda\Big|\leq C_N(1+|t|)^{-N}2^{j(1-K-2\sigma)}e^{-(n-1)d(z,z')/2}
\end{split}
\end{equation}
which implies that \eqref{b} is bounded by the right hand side of \eqref{dispersiveh} and \eqref{dispersiveh'}.\vspace{0.2cm}

Next we estimate \eqref{a}. Due the property of $a_\pm$, we divide it into two cases.

$\bullet$ Case 1: $d(z,z')\leq 1$. By using \eqref{beans}, we obtain 
\begin{equation*}
\begin{split}
&\Big|\int_0^\infty e^{it\lambda}\phi(2^{-j}\lambda) \lambda^{n-1-2\sigma}e^{\pm i\lambda d(z,z')}a_{\pm}(\lambda; z,z')
d\lambda\Big|\\&= \Big|\int_0^\infty \left(\frac1{
i(t-d(z,z'))}\frac\partial{\partial\lambda}\right)^{N}\big(e^{i(t-d(z,z'))\lambda}\big)
\phi(2^{-j}\lambda)\lambda^{n-1-2\sigma}a_\pm(\lambda;z,z') d\lambda\Big|\\&
\leq
C_N|t-d(z,z')|^{-N}\int_{2^{j-1}}^{2^{j+1}}\lambda^{n-1-2\sigma-N}(1+\lambda
d(z,z'))^{-\frac{n-1}2}d\lambda\\&\leq
C_N2^{j(n-2\sigma-N)}|t-d(z,z')|^{-N}(1+2^jd(z,z'))^{-(n-1)/2}.
\end{split}
\end{equation*}

$\bullet$ Case 2: $d(z,z')\geq 1$. By using \eqref{beans} again, we obtain 
\begin{equation*}
\begin{split}
&\Big|\int_0^\infty e^{it\lambda}\phi(2^{-j}\lambda) \lambda^{n-1-2\sigma}e^{\pm i\lambda d(z,z')}a_{\pm}(\lambda; z,z')
d\lambda\Big|\\&= \Big|\int_0^\infty \left(\frac1{
i(t-d(z,z'))}\frac\partial{\partial\lambda}\right)^{N}\big(e^{i(t-d(z,z'))\lambda}\big)
\phi(2^{-j}\lambda)\lambda^{n-1-2\sigma}a_\pm(\lambda;z,z') d\lambda\Big|\\&
\leq
C_N|t-d(z,z')|^{-N} e^{-(n-1)d(z,z')/2}\int_{2^{j-1}}^{2^{j+1}}\lambda^{n-1-2\sigma-N}\lambda^{-\frac{n-1}2} 
 d\lambda\\&\leq
C_N2^{j(n-2\sigma-N)}|t-d(z,z')|^{-N}2^{-j(n-1)/2}e^{-(n-1)d(z,z')/2}.
\end{split}
\end{equation*}
It follows that, for $d(z,z')\leq 1$
\begin{equation}\label{dispersive2}
\begin{split}
&\Big|\int_0^\infty e^{it\lambda} \phi(2^{-j}\lambda)(1-\chi)^2(\lambda)\lambda^{-2\sigma}\big(Q(\lambda)
E'_{\sqrt{H}}(\lambda)Q^*(\lambda)\big)(z,z')
d\lambda\Big|\\&\leq
C_N2^{j(n-2\sigma)}\big(1+2^j|t-d(z,z')|\big)^{-N}(1+2^jd(z,z'))^{-(n-1)/2},
\end{split}
\end{equation}
and for  $d(z,z')\geq 1$
\begin{equation}\label{dispersive2'}
\begin{split}
&\Big|\int_0^\infty e^{it\lambda} \phi(2^{-j}\lambda)(1-\chi)^2(\lambda)\lambda^{-2\sigma}\big(Q(\lambda)
E'_{\sqrt{H}}(\lambda)Q^*(\lambda)\big)(z,z')
d\lambda\Big|\\&\leq
C_N2^{j(n-2\sigma)}\big(1+2^j|t-d(z,z')|\big)^{-N}2^{-j(n-1)/2}e^{-(n-1)d(z,z')/2}.
\end{split}
\end{equation}
Consider the case $|t|\leq 2$. We first consider the case $d(z,z')\leq 1$. If $|t|\sim d(z,z')$, it is clear to see \eqref{dispersiveh}.
Otherwise, we have $|t-d(z,z')|\geq c|t|$ for some small constant
$c$, then choose $N=(n-1)/2$ to prove \eqref{dispersiveh}. For the second case $d(z,z')\geq 1$, by using $j\geq 0$, it follows from the fact $2^{-j}+|t|\lesssim 1$.
Therefore we have proved \eqref{dispersiveh}.\vspace{0.2cm}

Next we consider the case $|t|\geq 2$. We first consider the case $d(z,z')\leq 1$. Since $|t|\geq2$, we have $|t-d(z,z')|\geq \frac14|t|$, then by \eqref{dispersive2} for any $N$
\begin{equation*}
\begin{split}
\Big|\int_0^\infty e^{it\lambda} \phi(2^{-j}\lambda)(1-\chi)^2(\lambda)\lambda^{-2\sigma}&\big(Q(\lambda)
E'_{\sqrt{H}}(\lambda)Q^*(\lambda)\big)(z,z')
d\lambda\Big|\\&
\leq C_N2^{j(n-2\sigma)}2^{-Nj}|t|^{-N}.
\end{split}
\end{equation*}
 For the second case $d(z,z')\geq 1$, if $|t|\sim d(z,z')$, it is clear to see  for $0<\delta\ll1$
 \begin{equation*}
\begin{split}
&\Big|\int_0^\infty e^{it\lambda} \phi(2^{-j}\lambda)(1-\chi)^2(\lambda)\lambda^{-2\sigma}\big(Q(\lambda)
E'_{\sqrt{H}}(\lambda)Q^*(\lambda)\big)(z,z')
d\lambda\Big|\\&\leq
C_N2^{j[(n+1)/2-2\sigma]}e^{-(n-1)d(z,z')/2}\leq C_{N,\delta}2^{j[(n+1)/2-2\sigma]}|t|^{-N}e^{-(\rho-\delta)d(z,z')}.
\end{split}
\end{equation*}
Otherwise, we have $|t-d(z,z')|\geq c|t|$ for some small constant
$c$, then by \eqref{dispersive2'} for any $N$
\begin{equation*}
\begin{split}
\Big|\int_0^\infty e^{it\lambda} \phi(2^{-j}\lambda)(1-\chi)^2(\lambda)\lambda^{-2\sigma}&\big(Q(\lambda)
E'_{\sqrt{H}}(\lambda)Q^*(\lambda)\big)(z,z')
d\lambda\Big|
\\&\leq C_N2^{j[(n+1)/2-2\sigma]}2^{-Nj}|t|^{-N}e^{-(n-1)d(z,z')/2}.
\end{split}
\end{equation*}
By using the fact $j\geq 0$, therefore we have proved \eqref{dispersiveh'}.
\vspace{0.2cm}

We next prove \eqref{dispersivel'} which is for the low frequency i.e $j\leq 0$ and for any $0\leq\sigma<3/2$. 

$\bullet$ Case 1: $|t|\lesssim 1$. In this case,  we know from \eqref{lbeans'} that
\begin{equation}
\begin{split}
\Big|\int_0^\infty e^{it\lambda}\phi(2^{-j}\lambda)\lambda^{-2\sigma}&\chi^2(\lambda)
E'_{\sqrt{H}}(\lambda; z,z')
d\lambda\Big|\\&\leq C \int_{2^{j-1}}^{2^{j+1}} \phi(2^{-j}\lambda)\lambda^{2-2\sigma}
(1+d(z,z'))e^{-(n-1)d(z,z')/2}d\lambda\\&\leq C 2^{j(3-2\sigma)} (1+d(z,z'))e^{-(n-1)d(z,z')/2}
\end{split}
\end{equation}
which implies  \eqref{dispersivel'} when $|t|\lesssim 1$.\vspace{0.2cm}

$\bullet$ Case 2: $|t|\gg 1$. In this case, we further consider two subcases.

$\bullet$ Subcase 1: $ |t|\leq 2 d(z,z')$. In this subcase, arguing as above, we obtain
\begin{equation}
\begin{split}
\Big|\int_0^\infty e^{it\lambda}\phi(2^{-j}\lambda)\lambda^{-2\sigma}&\chi^2(\lambda)
E'_{\sqrt{H}}(\lambda; z,z')
d\lambda\Big|\\&\leq C \int_{2^{j-1}}^{2^{j+1}} \phi(2^{-j}\lambda)\lambda^{2-2\sigma}
(1+d(z,z'))e^{-(n-1)d(z,z')/2}d\lambda\\&\leq C 2^{j(3-2\sigma)}(1+d(z,z'))e^{-(n-1)d(z,z')/2}\\&\leq C2^{j(3-2\sigma)}|t|^{-N} e^{-(\rho-\delta)d(z,z')}.
\end{split}
\end{equation}
for any arbitrary large $N>0$ and $0<\delta\ll1$.

$\bullet$ Subcase 2: $ |t|\geq 2 d(z,z'), |t|\gg1$.  To show \eqref{dispersivel'}, it suffices to show, for $0<\delta\ll1$
\begin{equation}\label{dispersive''}
\begin{split}
\left|\int_0^\infty e^{it\lambda}\phi(2^{-j}\lambda)\lambda^{-2\sigma}\chi^2(\lambda)
E'_{\sqrt{H}}(\lambda; z,z')
d\lambda\right|\lesssim 2^{\mp\epsilon j}|t|^{2\sigma-3\mp\epsilon}e^{-(\rho-\delta)d(z,z')}.
\end{split}
\end{equation}
To this end, let $\bar{\lambda}=\lambda/t$ and recall $\sum_k\varphi(2^{-k}\lambda)=1$, we write
\begin{equation}
\begin{split}
&\int_0^\infty e^{it\lambda}\phi(2^{-j}\lambda)\lambda^{-2\sigma}\chi^2(\lambda)
E'_{\sqrt{H}}(\lambda; z,z')
d\lambda\\&=t^{2\sigma-1} \int_0^\infty e^{i\lambda} \lambda^{-2\sigma} \phi(2^{-j}\bar{\lambda}) \chi^2(\bar{\lambda})
E'_{\sqrt{H}}(\bar{\lambda};z,z')d\lambda \\&=t^{2\sigma-1} \sum_{k\in\Z}\int_0^\infty e^{i\lambda} \lambda^{-2\sigma}\varphi(2^{-k}\lambda) \phi(2^{-j}\bar{\lambda}) \chi^2(\bar{\lambda})
E'_{\sqrt{H}}(\bar{\lambda};z,z')d\lambda.
\end{split}
\end{equation}
Define
\begin{equation}
\begin{split}
&I=t^{2\sigma-1} \sum_{k\leq 0}\int_0^\infty e^{i\lambda} \lambda^{-2\sigma}\varphi(2^{-k}\lambda) \phi(2^{-j}\bar{\lambda}) \chi^2(\bar{\lambda})
E'_{\sqrt{H}}(\bar{\lambda};z,z')d\lambda; \\
&II=t^{2\sigma-1} \sum_{k\geq 1}\int_0^\infty e^{i\lambda} \lambda^{-2\sigma}\varphi(2^{-k}\lambda) \phi(2^{-j}\bar{\lambda}) \chi^2(\bar{\lambda})
E'_{\sqrt{H}}(\bar{\lambda};z,z')d\lambda.\\
\end{split}
\end{equation}
Recall $\bar{\lambda}=\lambda/t$, by \eqref{lbeans'} and $\lambda/t\sim 2^{j}$, then we have 
\begin{equation}
\begin{split}
&|I|=|t|^{2\sigma-1} \left|\sum_{k\leq 0}\int_0^\infty e^{i\lambda} \lambda^{-2\sigma}\varphi(2^{-k}\lambda) \phi(2^{-j}\bar{\lambda}) \chi^2(\bar{\lambda})
E'_{\sqrt{H}}(\bar{\lambda};z,z')d\lambda\right| \\
&\leq t^{2\sigma-1}2^{\mp j\epsilon}\sum_{k\leq 0}\int_{2^k}^{2^{k+1}} \lambda^{-2\sigma}(t^{-1}\lambda)^{2\pm\epsilon}(1+d(z,z'))e^{-(n-1)d(z,z')/2} d\lambda \\
&\lesssim 2^{\mp j\epsilon} t^{2\sigma-3\mp\epsilon} (1+d(z,z')) e^{-(n-1)d(z,z')/2},\quad 0\leq \sigma<3/2,  0\leq\epsilon\ll\min\{1,3-2\sigma\} \end{split}
\end{equation}
which gives \eqref{dispersive''}.
By \eqref{lbeans} in Proposition \ref{prop:localized spectral measure}, we have 
\begin{equation}
\begin{split}
II&=t^{2\sigma-1} \sum_{k\geq 1}\int_0^\infty e^{i\lambda} \lambda^{-2\sigma}\varphi(2^{-k}\lambda) \phi(2^{-j}\bar{\lambda}) \chi^2(\bar{\lambda})
E'_{\sqrt{H}}(\bar{\lambda};z,z')d\lambda\\
&= 2^{\mp \epsilon j}t^{2\sigma-1}(\rho_L\rho_R)^{\frac {n-1}2}\sum_{k\geq 1}\int_0^\infty e^{i\lambda} (t^{-1}\lambda)^{1\pm\epsilon}\lambda^{-2\sigma}\varphi(2^{-k}\lambda) \phi(2^{-j}\bar{\lambda})\chi^2(\bar{\lambda}) \\
&\qquad \times\left((\rho_L\rho_R)^{i\bar{\lambda}}a(\bar{\lambda};z,z')-(\rho_L\rho_R)^{-i\bar{\lambda}}a(-\bar{\lambda};z,z')\right) d\lambda.
\end{split}
\end{equation}
By integration by parts, we estimate
\begin{equation}
\begin{split}
II &\lesssim  t^{2\sigma-2\mp\epsilon}2^{\mp j\epsilon}(\rho_L\rho_R)^{\frac {n-1}2}\sum_{k\geq 1}\int_0^\infty \big(\frac{d}{ d\lambda}\big)^4 \Big(\lambda^{1\pm\epsilon-2\sigma}\varphi(2^{-k}\lambda) \phi(2^{-j}\bar{\lambda})\chi^2(\bar{\lambda}) \\
&\qquad \times\left((\rho_L\rho_R)^{i\bar{\lambda}}a(\bar{\lambda};z,z')-(\rho_L\rho_R)^{-i\bar{\lambda}}a(-\bar{\lambda};z,z')\right)\Big) d\lambda.
\end{split}
\end{equation}
If none of the derivative hits the term $\left((\rho_L\rho_R)^{i\bar{\lambda}}a(\bar{\lambda};z,z')-(\rho_L\rho_R)^{-i\bar{\lambda}}a(-\bar{\lambda};z,z')\right)$,
since $|\bar{\lambda}|=|\lambda/t|\leq 1$, then we use the smoothness of $a$ at $0$ to obtain 
$$\left((\rho_L\rho_R)^{i\bar{\lambda}}a(\bar{\lambda};z,z')-(\rho_L\rho_R)^{-i\bar{\lambda}}a(-\bar{\lambda};z,z')\right)\lesssim\bar{\lambda}\lesssim \lambda t^{-1}.$$
If the derivatives hit the other terms we gain $\lambda^{-4}$. In this case, note that $0 \leq \sigma<3/2$ and $0\leq \epsilon\ll1$ and we show 
\begin{equation}
\begin{split}
|II_1|&\lesssim  t^{2\sigma-3\mp\epsilon}2^{\mp j\epsilon}(\rho_L\rho_R)^{\frac {n-1}2}\sum_{k\geq 1}\int_{2^{k}}^{2^{k+1}} \lambda^{-2\mp\epsilon}d\lambda\lesssim  t^{2\sigma-3\mp\epsilon}2^{\mp j\epsilon}(\rho_L\rho_R)^{\frac {n-1}2}.
\end{split}
\end{equation}
If at least one derivative hits the term $\left((\rho_L\rho_R)^{i\bar{\lambda}}a(\bar{\lambda};z,z')-(\rho_L\rho_R)^{-i\bar{\lambda}}a(-\bar{\lambda};z,z')\right)$, since $a\in \mathcal{C}^\infty$, we gain $t^{-1}$ at least.
Note that $\lambda/t\lesssim 1$, we gain in total  $\lambda^{-3}t^{-1}$, then 
\begin{equation}
\begin{split}
|II_2|&\lesssim  t^{2\sigma-3\mp\epsilon}2^{\mp j\epsilon}(\rho_L\rho_R)^{\frac {n-1}2}(\ln(\rho_L\rho_R))^4\sum_{k\geq 1}\int_{2^{k}}^{2^{k+1}} \lambda^{-2\pm\epsilon}d\lambda
\\&\lesssim  t^{2\sigma-3\mp\epsilon}2^{\mp j\epsilon}(\rho_L\rho_R)^{\frac {n-1}2}(\ln(\rho_L\rho_R))^4.
\end{split}
\end{equation}
From \cite[Proposition 3.4]{CH2}, we have
$$(\rho_L\rho_R)^{\frac {n-1}2}(\ln(\rho_L\rho_R))^4\leq (1+d(z,z'))^4 e^{-(n-1)d(z,z')/2}.$$
Therefore we prove \eqref{dispersive''}, hence we have \eqref{dispersivel'}. The proof of Lemma \ref{lem:dispersive} is then complete.

\end{proof}

\subsection{Microlocalized Strichartz estimate} In this subsection, we use the $L^2$-estimate and dispersive estimate for the microlocalized wave propagator to obtain 
the microlocalized Strichartz estimate.

\begin{proposition}\label{m-strichartz}
Let  $\leftidx{^{\sigma}}U^{\mathrm{low}}_j(t) $ and $\leftidx{^{\sigma}}U^{\mathrm{high}}_{j,k}(t) $ be defined in \eqref{Ul-mic} and  \eqref{Uh-mic} and $n\geq3$.
Then for every pair $(q,r)\in[2,\infty]\times(2,\infty]$, there exists a constant $C$ only depending on $n$,
$q$ and $r$ such that:\vspace{0.1cm}

$\bullet$ For $j\geq0$ and $\sigma\geq0$
\begin{equation}\label{stri}
\Big(\int_{\R}\|\leftidx{^{\sigma}}U^{\mathrm{high}}_{j,k}(t) f\|_{L^r}^q dt\Big)^{\frac1q}\leq C (1+j)
2^{j(s-\sigma)}\|f\|_{L^2},
\end{equation}where $s=s_e$ as in \eqref{se} when $2/q\geq (n-1)(1/2-1/r)$ and $s=s_w$ defined in \eqref{sw} when $2/q\leq (n-1)(1/2-1/r)$;\vspace{0.1cm}

$\bullet$ For $j\leq0$, if $0\leq \sigma<1$ 
\begin{equation}\label{stri'}
\Big(\int_{\R}\|\leftidx{^{\sigma}}U^{\mathrm{low}}_j(t)  f\|_{L^r}^q dt\Big)^{\frac1q}\leq C
2^{\epsilon j}\|f\|_{L^2}; \quad \forall 0\leq\epsilon\ll 1,\quad \epsilon<1-\sigma;
\end{equation}

$\bullet$ For $j\leq0$, $\sigma=1$ and $q\geq 2$
\begin{equation}\label{stri''}
\Big(\int_{\R}\|\leftidx{^{\sigma}}U^{\mathrm{low}}_j(t)  f\|_{L^r}^q dt\Big)^{\frac1q}\leq C
2^{-j\epsilon}\|f\|_{L^2},\quad \forall 0<\epsilon\ll 1.
\end{equation}
In addition if $q\neq 2$, one can choose $\epsilon=0$.

\end{proposition}

\begin{remark} The log regularity $j$ in \eqref{stri} appears on the line $\frac2q= (n-1)(\frac12-\frac1r)$.
This loss can be removed using Keel-Tao's argument
\cite[Sections 3-7]{KT}, but we do not pursue here sharp regularity. 
\end{remark}

\begin{proof} We closely follow Keel-Tao's argument
\cite[Sections 3-7]{KT}. By the $TT^*$ argument, it
suffices to show
\begin{equation}\label{bi-est}
\begin{split}
\Big|\iint\langle (\leftidx{^{\sigma}}U^{\mathrm{high}}_{j,k}(\tau))^*F(\tau), (\leftidx{^{\sigma}}U^{\mathrm{high}}_{j,k}(t))^*G(t) \rangle d\tau dt\Big|\lesssim
2^{2j(s-\sigma)}(1+j)^2\|F\|_{L^{q'}_tL^{r'}}\|G\|_{L^{q'}_tL^{r'}},
\end{split}
\end{equation}
and
\begin{equation}\label{bi-est'}
\begin{split}
\Big|\iint\langle (\leftidx{^{\sigma}}U^{\mathrm{low}}_{j}(\tau))^*F(\tau), (\leftidx{^{\sigma}}U^{\mathrm{low}}_{j}(t))^*G(t) \rangle d\tau dt\Big|\lesssim
C\Lambda(j)^2\|F\|_{L^{q'}_tL^{r'}}\|G\|_{L^{q'}_tL^{r'}},
\end{split}
\end{equation}
where $\Lambda(j)= 2^{\epsilon j}$ when $0\leq\sigma<1$ with $0\leq \epsilon\ll1$, and $\Lambda(j)= 2^{-\epsilon j}$ when $\sigma=1$ with $0<\epsilon\ll1$. 
In particular, $\sigma=1$ and $q\neq 2$ one can choose $\epsilon=0$.\vspace{0.2cm}

To this end, we consider four cases.

{\bf Case 1:} $j\geq0$ and $|t-\tau|\leq 2$. By the interpolation of the bilinear form of \eqref{Dispersiveh} and the energy estimate in Proposition \ref{energy}, we have
\begin{equation*}
\begin{split}
&\langle (\leftidx{^{\sigma}}U^{\mathrm{high}}_{j,k}(\tau))^*F(\tau), (\leftidx{^{\sigma}}U^{\mathrm{high}}_{j,k}(t))^*G(t) \rangle\\&\leq
C2^{j[(n+1)(\frac12-\frac1r)-2\sigma]}(2^{-j}+|t-\tau|)^{-(n-1)(\frac12-\frac1r)}\|F\|_{L^{r'}}\|G\|_{L^{r'}}.
\end{split}
\end{equation*}
Therefore we obtain by H\"older's and Young's inequalities 
\begin{equation*}
\begin{split}
&\Big|\iint\langle (\leftidx{^{\sigma}}U^{\mathrm{high}}_{j,k}(\tau))^*F(\tau), (\leftidx{^{\sigma}}U^{\mathrm{high}}_{j,k}(t))^*G(t) \rangle
dsdt\Big|\\&\lesssim
2^{j[(n+1)(\frac12-\frac1r)-2\sigma]}\iint_{|t-\tau|\leq 2}(2^{-j}+|t-\tau|)^{-(n-1)(\frac12-\frac1r)}\|F(\tau)\|_{L^{r'}}\|G(t)\|_{L^{r'}}dtd\tau\\&
\lesssim
2^{2j(s_e-\sigma)}\max\{2^{j[(n-1)(\frac12-\frac1r)-\frac2q]},1\}\|F\|_{L^{q'}_tL^{r'}}\|G\|_{L^{q'}_tL^{r'}}
\end{split}
\end{equation*}
when $\frac2q\neq (n-1)(\frac12-\frac1r)$. If $\frac2q= (n-1)(\frac12-\frac1r)$, we similarly have 
\begin{equation*}
\begin{split}
&\Big|\iint\langle (\leftidx{^{\sigma}}U^{\mathrm{high}}_{j,k}(\tau))^*F(\tau), (\leftidx{^{\sigma}}U^{\mathrm{high}}_{j,k}(t))^*G(t) \rangle
dsdt\Big|\\&\lesssim
(1+j)2^{2j(s_e-\sigma)}\|F\|_{L^{q'}_tL^{r'}}\|G\|_{L^{q'}_tL^{r'}}.
\end{split}
\end{equation*}

{\bf Case 2:} $j\geq0$ and $|t-\tau|\geq 2$. Similarly, by the interpolation of the bilinear form of \eqref{Dispersiveh'} and the energy estimate in Proposition \ref{energy}, we have
\begin{equation*}
\begin{split}
\langle (\leftidx{^{\sigma}}U^{\mathrm{high}}_{j,k}(\tau))^*F(\tau), (\leftidx{^{\sigma}}U^{\mathrm{high}}_{j,k}(t))^*G(t) \rangle&\leq
C2^{j[(n+1)(\frac12-\frac1r)-2\sigma]}|t-\tau|^{-2N(\frac12-\frac1r)}\|F\|_{L^{r'}}\|G\|_{L^{r'}}.
\end{split}
\end{equation*}
Therefore,  by using H\"older's and Young's inequalities and choosing $N$ enough, we obtain
\begin{equation*}
\begin{split}
&\Big|\iint\langle (\leftidx{^{\sigma}}U^{\mathrm{high}}_{j,k}(\tau))^*F(\tau), (\leftidx{^{\sigma}}U^{\mathrm{high}}_{j,k}(t))^*G(t) \rangle
dsdt\Big|\\&\lesssim
2^{j[(n+1)(\frac12-\frac1r)-2\sigma]}\iint_{|t-\tau|\geq 2}|t-\tau|^{-2N(\frac12-\frac1r)}\|F(\tau)\|_{L^{r'}}\|G(t)\|_{L^{r'}}dtd\tau\\&
\lesssim
2^{2j(s_e-\sigma)}\|F\|_{L^{q'}_tL^{r'}}\|G\|_{L^{q'}_tL^{r'}}.
\end{split}
\end{equation*}
By the definition of $s$, we collect the two cases to prove \eqref{stri}.\vspace{0.1cm}

{\bf Case 3:} $j\leq0$ and $0\leq \sigma<1$.  By using \eqref{dispersivel'} with positive sign and small $\delta$ satisfying $0<\delta<\delta_0(r)$ as in Lemma \ref{lem:qq'}, 
we use 
 \eqref{est:qq'} to obtain
\begin{equation*}
\begin{split}
\langle (\leftidx{^{\sigma}}U^{\mathrm{low}}_{j}(\tau))^*F(\tau), (\leftidx{^{\sigma}}U^{\mathrm{low}}_{j}(t))^*G(t) \rangle&\leq
C\|\leftidx{^{\sigma}}U^{\mathrm{low}}_{j}(t)(\leftidx{^{\sigma}}U^{\mathrm{low}}_{j}(\tau))^*F\|_{L^{r}}\|G(t)\|_{L^{r'}}\\&\leq C 2^{2\epsilon j}(1+|t-\tau|)^{2\sigma-3+2\epsilon}
\big\|\int e^{-(\rho-\delta)d(z,z')}F dg(z')\big\|_{L^{r}}\|G(t)\|_{L^{r'}}
\\&\leq C2^{2\epsilon j}(1+|t-\tau|)^{2\sigma-3+2\epsilon} \|F(\tau)\|_{L^{r'}}\|G(t)\|_{L^{r'}}.
\end{split}
\end{equation*}
Note that if  $0\leq \sigma<1$,  for $q\geq2$,  it gives $2/q<3-2\sigma-2\epsilon$ when $0\leq \epsilon\ll1-\sigma$ . Therefore,  by using H\"older's and Young's inequalities, we obtain for $q\geq2$
\begin{equation*}
\begin{split}
&\Big|\iint\langle (\leftidx{^{\sigma}}U^{\mathrm{low}}_{j}(\tau))^*F(\tau), (\leftidx{^{\sigma}}U^{\mathrm{low}}_{j}(t))^*G(t) \rangle
d\tau dt\Big|\\&\lesssim
2^{2\epsilon j}\iint(1+|t-\tau|)^{2\sigma-3+2\epsilon}\|F(t)\|_{L^{r'}}\|G(\tau)\|_{L^{r'}}dtd\tau\\&
\lesssim 2^{2\epsilon j} \|F\|_{L^{q'}_tL^{r'}}\|G\|_{L^{q'}_tL^{r'}}.
\end{split}
\end{equation*}
This proves \eqref{bi-est'}. 

{\bf Case 4:} $j\leq0$, $\sigma=1$ and $q\geq2$. By using \eqref{dispersivel'} with negative sign and similar argument as above, we have
\begin{equation*}
\begin{split}
\langle (\leftidx{^{\sigma}}U^{\mathrm{low}}_{j}(\tau))^*F(\tau), (\leftidx{^{\sigma}}U^{\mathrm{low}}_{j}(t))^*G(t) \rangle&\leq
C\|\leftidx{^{\sigma}}U^{\mathrm{low}}_{j}(t)(\leftidx{^{\sigma}}U^{\mathrm{low}}_{j}(\tau))^*F\|_{L^{r}}\|G(t)\|_{L^{r'}}\\&\leq C2^{-2j\epsilon}(1+|t-\tau|)^{-1-2\epsilon}\big\|\int e^{-(\rho-\delta)d(z,z')}F dg(z')\big\|_{L^{r}}\|G(t)\|_{L^{r'}}
\\&\leq C2^{-2j\epsilon}(1+|t-\tau|)^{-1-2\epsilon}\|F(\tau)\|_{L^{r'}}\|G(t)\|_{L^{r'}}.
\end{split}
\end{equation*}
So similar above argument proves \eqref{bi-est'}.  In particular $q>2$, it is clear that one can choose $\epsilon=0$.

\end{proof}

\section{Homogeneous Strichartz estimates}
In this section, we prove Theorem \ref{thm:hstri} by using the mircolocalized Strichartz estimate in Proposition \ref{m-strichartz}.
Recall $H=-\Delta_g-\rho^2$ and let $u$ be the solution of 
\begin{equation}\label{leq}
\partial_{t}^2u+H u=0, \quad u(0)=u_0(z),
~\partial_tu(0)=u_1(z).
\end{equation}
Then we have
\begin{equation*}
\begin{split}
u(t,z)
=\frac{U(t)+U(-t)}2 u_0+\frac{U(t)-U(-t)}{2i\sqrt{H}}u_1,
\end{split}
\end{equation*}
where $U(t)=e^{it\sqrt{H}}$.
By recalling $\leftidx{^{\sigma}}U(t)=e^{it\sqrt{H}}H^{-\frac{\sigma}2}$ and  using \eqref{Ulh}, we aim to estimate
\begin{equation}
\begin{split}
&\|u\|_{L^q(\R;L^r(M^\circ))}\\&\lesssim
\sum_{\pm}\sum_{\sigma\in\{0,1\}}\left(\|\leftidx{^{\sigma}}U^{\mathrm{low}}(\pm t)u_\sigma\|_{L^q(\R;L^r(M^\circ))}+\|\leftidx{^{\sigma}}U^{\mathrm{high}}(\pm t)u_\sigma\|_{L^q(\R;L^r(M^\circ))}\right).
\end{split}
\end{equation}
 To prove \eqref{hstri} in Theorem \ref{thm:hstri},  it is enough to prove
 \begin{equation}\label{hstri'}
\|\leftidx{^{\alpha}}U^{\mathrm{high}}(t)f\|_{L^q_t(\R:L^r(M^\circ))}\lesssim
 \|f\|_{L^2(M^\circ)}, 
\end{equation}
with $\alpha=\mu$
and
\begin{equation}\label{lstri'}
\|\leftidx{^{\beta}}U^{\mathrm{low}}(t)f\|_{L^q_t(\R:L^r(M^\circ))}\lesssim \|f\|_{L^2(M^\circ)},
\end{equation}
 where $\beta$ equals $0$ or $1-\epsilon$ with $0<\epsilon\ll1$.
  Recall $\leftidx{^{\sigma}}U^{\mathrm{low}}_j(t) $ and $\leftidx{^{\sigma}}U^{\mathrm{high}}_{j,k}(t) $ defined in \eqref{Ul-mic} and  \eqref{Uh-mic}, then we have
\begin{equation*}
\begin{split}
\leftidx{^{\sigma}}U^{\mathrm{high}}(t)f=\sum_{j\geq0}\sum_{k=0}^{N}\leftidx{^{\sigma}}U^{\mathrm{high}}_{j,k}(t)f
\end{split}
\end{equation*}
and
\begin{equation*}
\begin{split}
\leftidx{^{\sigma}}U^{\mathrm{low}}(t)f=\sum_{j\leq0}\leftidx{^{\sigma}}U_j^{\mathrm{low}}f.
\end{split}
\end{equation*}
By using Proposition \ref{m-strichartz} with $\sigma=\alpha=\mu$, hence we obtain for $j\geq0$, $0\leq k\leq N$
\begin{equation*}
\|\leftidx{^{\alpha}}U^{\mathrm{high}}_{j,k}(t)f\|_{L^q_t(\R:L^r(M^\circ))}\lesssim
2^{j(s-\mu)} \|f\|_{L^2(M^\circ)},\quad s=s_e, ~s_w.
\end{equation*}
Note that $\mu>s$ and take summation in $j\geq0$ and finite $k$, we prove \eqref{hstri'}.

If $\beta=0$, by using Proposition \ref{m-strichartz} with $\sigma=0$, we obtain for $j\leq0$, 
\begin{equation*}
\|\leftidx{^{\beta}}U_j^{\mathrm{low}}(t)f\|_{L^q_t(\R:L^r(M^\circ))}\lesssim 2^{\epsilon j}\|f\|_{L^2(M^\circ)},
\end{equation*}
 and if $\beta=1-\epsilon$ with $0< \epsilon\ll1$, choose $0<\tilde{\epsilon}<\epsilon=1-\beta$, we use \eqref{stri'} in Proposition \ref{m-strichartz} with $\sigma=\beta=1-\epsilon$ to obtain 
\begin{equation*}
\|\leftidx{^{\beta}}U_j^{\mathrm{low}}(t)f\|_{L^q_t(\R:L^r(M^\circ))}\lesssim
 2^{\tilde{\epsilon} j}\|f\|_{L^2(M^\circ)}.
\end{equation*}
By summing in $j\leq0$, we obtain \eqref{lstri'} with $\beta=0$ and $1-\epsilon$. Hence we have proved \eqref{hstri} in Theorem \ref{thm:hstri}.

\section{Inhomogeneous Strichartz estimates} 

In this section, we prove the inhomogeneous Strichartz estimate in Theorem \ref{thm:ihstri}. To this purpose, we divide into two cases. The first case that $q>\tilde{q}'$ is much 
easier to prove due to the Christ-Kiselev lemma \cite{CK}. The second case when $q=\tilde{q}=2$ is more complicated since the usual dispersive estimate fails due to
the conjugate points. We call the inhomogeneous Strichartz estimate as double endpoint estimate when $q=\tilde{q}=2$, otherwise we call them non-double endpoint inhomogeneous Strichartz estimate. 

\subsection{Inhomogeneous Strichartz estimates for non-double endpoint}
In this subsection, we prove 
\begin{proposition}\label{prop:inh} Let $(q,r,\mu), (\tilde{q},\tilde{r},\tilde{\mu})\in \Lambda_w\cup \Lambda_e $ and at least one of $q,\tilde{q}$ does not equal to $2$,  
the following inequalities hold:\vspace{0.2cm}

$\bullet$ Low frequency estimate
\begin{equation}\label{l-inh}
\begin{split}
\Big\|\int_{\tau<t}\frac{\sin{(t-\tau)\sqrt{H}}}
{\sqrt{H}}\chi(\sqrt{H})F(\tau)d\tau\Big\|_{L^q_tL^{r}_z}\lesssim \|F\|_{L^{\tilde{q}'}_t{L}^{\tilde{r}'}_z},
\end{split}
\end{equation}

$\bullet$ High frequency estimate
\begin{equation}\label{h-inh}
\begin{split}
\Big\|\int_{\tau<t}\frac{\sin{(t-\tau)\sqrt{H}}}
{\sqrt{H}}(1-\chi)(\sqrt{H})F(\tau)d\tau\Big\|_{L^q_tL^{r}_z}\lesssim \|H^{\frac{\mu+\tilde{\mu}-1}2}F\|_{L^{\tilde{q}'}_t{L}^{\tilde{r}'}_z},
\end{split}
\end{equation}
where $\chi\in\mathcal{C}_c^\infty([0,\infty)$ such that $\chi(\lambda) = 1$ for $\lambda \leq 1$ and vanishes when $\lambda\geq2$.  

\end{proposition}

\begin{remark} We can obtain a special inhomogeneous Strichartz estimate that we shall require in the next section. For $p\in(1, 1+4/(n-1))$, we have 
\begin{equation}\label{(p+1)-inh}
\begin{split}
\Big\|\int_{\tau<t}\frac{\sin{(t-\tau)\sqrt{H}}}
{\sqrt{H}}F(\tau)d\tau\Big\|_{L^{p+1}_tL^{p+1}_z}\lesssim \|F\|_{L^{\frac{p+1}p}_t{L}^{\frac{p+1}p}_z}.
\end{split}
\end{equation}
Indeed, the low frequency part follows from \eqref{l-inh}. Choose $\mu=\tilde{\mu}=1/2$, then we can check 
$$(p+1,p+1,1/2)\in \Lambda_e, \quad p+1>2$$
when $p\in(1, 1+4/(n-1))$. Hence the high frequency part follows from \eqref{h-inh}.

\end{remark}

\begin{proof}[Proof of Proposition \ref{prop:inh}] We first prove \eqref{l-inh}.
Recall $U(t)=e^{it\sqrt{H}}$, then
\begin{equation*}
\begin{split}
\frac{\sin{(t-\tau)\sqrt{H}}}
{\sqrt{H}}\chi(\sqrt{H})&=H^{-\frac12}\chi(\sqrt{H})(U(t)U(\tau)^*-U(-t)U(-\tau)^*)/2i
\\&=\frac{1}{2i}\big(\leftidx{^{\sigma}}U^{\mathrm{low}}(t)(\leftidx{^{\sigma}}U^{\mathrm{low}}(\tau))^*-\leftidx{^{\sigma}}U^{\mathrm{low}}(-t)(\leftidx{^{\sigma}}U^{\mathrm{low}}(-\tau))^*\big),\  \sigma=1/2,
\end{split}
\end{equation*}
where 
\begin{equation}\label{Ul}
\leftidx{^{\sigma}}U^{\mathrm{low}}(t) = \int_0^\infty e^{it\lambda} \chi^{1/2}(\lambda)\lambda^{-\sigma}
dE_{\sqrt{H}}(\lambda).
\end{equation}
This is just the analog of \eqref{Ulh} with 
$\chi(\lambda)$ there replaced  by $\chi^{1/2}(\lambda)$, which causes no problems. 
Since the other term can be treated similarly,  it suffices to show that
\begin{equation}\label{l-inh'}
\begin{split}
\int_{\tau<t}\leftidx{^{\sigma}}U^{\mathrm{low}}(t)(\leftidx{^{\sigma}}U^{\mathrm{low}}(\tau))^*F(\tau)d\tau,\quad \sigma=1/2
\end{split}
\end{equation}
satisfies the bounds in \eqref{l-inh}.
As before, by using Proposition \ref{m-strichartz} with $\sigma=1/2$, we obtain for $j\leq0$, 
\begin{equation*}
\|\leftidx{^{\sigma}}U_j^{\mathrm{low}}(t)f\|_{L^q_t(\R:L^r(M^\circ))}\lesssim 2^{\epsilon j}\|f\|_{L^2(M^\circ)}, \quad 0\leq \epsilon\ll1,
\end{equation*}
and hence we further have 
\begin{equation*}
\|\leftidx{^{\sigma}}U^{\mathrm{low}}(t)f\|_{L^q_t(\R:L^r(M^\circ))}\lesssim \sum_{j\leq 0}\|\leftidx{^{\sigma}}U_j^{\mathrm{low}}(t)f\|_{L^q_t(\R:L^r(M^\circ))}\lesssim \|f\|_{L^2(M^\circ)}.
\end{equation*}
By the duality, we have the following
\begin{equation*}
\Big\|\int_\R \leftidx{^{\sigma}}U^{\mathrm{low}}(t)(\leftidx{^{\sigma}}U^{\mathrm{low}}(\tau))^*F(\tau) d\tau\Big\|_{L^q_t(\R:L^r(M^\circ))}\lesssim \|F\|_{L^{\tilde{q}'}_t(\R:L^{\tilde{r}'}(M^\circ))}.
\end{equation*}
Under the assumption that at least one of $q,\tilde{q}$ is not $2$, we have $q>\tilde{q}'$. Hence by using Christ-Kiselev lemma \cite{CK}, we obtain 
\begin{equation*}
\Big\|\int_{\tau<t} \leftidx{^{\sigma}}U^{\mathrm{low}}(t)(\leftidx{^{\sigma}}U^{\mathrm{low}}(\tau))^*F(\tau) d\tau\Big\|_{L^q_t(\R:L^r(M^\circ))}\lesssim \|F\|_{L^{\tilde{q}'}_t(\R:L^{\tilde{r}'}(M^\circ))}.
\end{equation*}
Therefore we have shown that \eqref{l-inh'} satisfies the bounds in \eqref{l-inh}, as desired.

Next we prove \eqref{h-inh}. Similarly as above, we write 
\begin{equation*}
\begin{split}
&\frac{\sin{(t-\tau)\sqrt{H}}}
{\sqrt{H}}H^{-\frac{\mu+\tilde{\mu}-1}2}(1-\chi)(\sqrt{H})\\&=H^{-\frac{\mu+\tilde{\mu}}2}(1-\chi)(\sqrt{H})(U(t)U(\tau)^*-U(-t)U(-\tau)^*)/2i
\\&=\frac{1}{2i}\big(\leftidx{^{\mu}}U^{\mathrm{high}}(t)(\leftidx{^{\tilde{\mu}}}U^{\mathrm{high}}(\tau))^*-\leftidx{^{\mu}}U^{\mathrm{high}}(-t)(\leftidx{^{\tilde{\mu}}}U^{\mathrm{high}}(-\tau))^*\big),
\end{split}
\end{equation*}
where 
\begin{equation}\label{Uh}
\leftidx{^{\sigma}}U^{\mathrm{high}}(t) = \int_0^\infty e^{it\lambda} (1-\chi)^{1/2}(\lambda)\lambda^{-\sigma}
dE_{\sqrt{H}}(\lambda).
\end{equation}
Here we have replaced $(1-\chi)(\lambda)$ in \eqref{Ulh} by $(1-\chi)^{1/2}(\lambda)$ which is inconsequential. 
To prove \eqref{h-inh}, it suffices to show
\begin{equation}\label{h-inh'}
\begin{split}
\Big\|\int_{\tau<t}\leftidx{^{\mu}}U^{\mathrm{high}}(t)(\leftidx{^{\tilde{\mu}}}U^{\mathrm{high}}(\tau))^*F(\tau)d\tau\Big\|_{L^q_tL^{r}_z}\lesssim \|F\|_{L^{\tilde{q}'}_t{L}^{\tilde{r}'}_z}.
\end{split}
\end{equation}
Applying Proposition \ref{m-strichartz} with $\sigma=\mu$ and its dual version with $\sigma=\tilde{\mu}$,
we have
for  all $j\geq0$ and $k=0,\cdots, N$
\begin{equation*}
\big\|\leftidx{^{\mu}}U^{\mathrm{high}}_{j,k}(t) f\big\|_{L^q_t(\R:L^r(M^\circ))}\lesssim
2^{j(s-\mu)} \|f\|_{L^2(M^\circ)},\quad s=s_e, s_w,
\end{equation*} 
and 
\begin{equation*}
\Big\|\int_{\R}(\leftidx{^{\tilde{\mu}}}U^{\mathrm{high}}_{j,k}(\tau))^* F(\tau) d\tau\Big\|_{L^2(M^\circ)}\lesssim
2^{j(s-\tilde{\mu})} \|F\|_{L^{\tilde{q}'}_t{L}^{\tilde{r}'}_z},\quad s=s_e, s_w.
\end{equation*} 
Therefore we obtain, for all $k,k'\in\{0,\cdots, N\}$ and $j,j'\geq0$
\begin{equation*}
\Big\|\int_\R \leftidx{^{\mu}}U^{\mathrm{high}}_{j,k}(t)(\leftidx{^{\tilde{\mu}}}U^{\mathrm{high}}_{j',k'}(\tau))^* F(\tau) d\tau\Big\|_{L^q_t(\R:L^r(M^\circ))}\lesssim
2^{j(s-\mu)} 2^{j'(s-\tilde{\mu})}\|F\|_{L^{\tilde{q}'}_t(\R:L^{\tilde{r}'}(M^\circ))}.
\end{equation*}
Let  
\begin{equation}\label{U>}
\leftidx{^{\sigma}}U^{\mathrm{high}}_{\geq,k}(t)=\sum_{j\geq0}\leftidx{^{\sigma}}U^{\mathrm{high}}_{j,k}(t),
\end{equation}
since $\mu,\tilde{\mu}>s$, then we sum over $j$ and $j'$ to show
\begin{equation}\label{est:U>}
\Big\|\int_\R \leftidx{^{\mu}}U^{\mathrm{high}}_{\geq,k}(t)(\leftidx{^{\tilde{\mu}}}U^{\mathrm{high}}_{\geq,k'}(\tau))^* F(\tau) d\tau\Big\|_{L^q_t(\R:L^r(M^\circ))}\lesssim
\|F\|_{L^{\tilde{q}'}_t(\R:L^{\tilde{r}'}(M^\circ))}.
\end{equation}
Further by taking the summation in $k,k'$ which range over a finite set and using Christ-Kiselev lemma with $q>\tilde{q}'$, we thus prove \eqref{h-inh'}.

\end{proof}

\subsection{Inhomogeneous Strichartz estimates on the double endpoint} We prove the following result on the double endpoint  inhomogeneous Strichartz estimate.
\begin{proposition}\label{prop:dinh} Let $(q,r,\mu), (\tilde{q},\tilde{r},\tilde{\mu})\in \Lambda_w\cup \Lambda_e $ and $q=\tilde{q}=2$,  the following inequalities hold:

$\bullet$ Low frequency estimate
\begin{equation}\label{l-dinh}
\begin{split}
\Big\|\int_{\tau<t}\frac{\sin{(t-\tau)\sqrt{H}}}
{\sqrt{H}}\chi(\sqrt{H})F(\tau)d\tau\Big\|_{L^2_tL^{r}_z}\lesssim \|F\|_{L^{2}_t{L}^{\tilde{r}'}_z},
\end{split}
\end{equation}

$\bullet$ High frequency estimate
\begin{equation}\label{h-dinh}
\begin{split}
\Big\|\int_{\tau<t}\frac{\sin{(t-\tau)\sqrt{H}}}
{\sqrt{H}}(1-\chi)(\sqrt{H})F(\tau)d\tau\Big\|_{L^2_tL^{r}_z}\lesssim \|H^{\frac{\mu+\tilde{\mu}-1}2}F\|_{L^{2}_t{L}^{\tilde{r}'}_z},
\end{split}
\end{equation}
where $\chi\in\mathcal{C}_c^\infty([0,\infty)$ such that $\chi(\lambda) = 1$ for $\lambda \leq 1$ and vanishes when $\lambda\geq2$.  

\end{proposition}

\begin{proof} The above argument breaks down here due to the failure of the Christ-Kiselev lemma.  We follow the argument in Keel-Tao \cite{KT} to overcome this obstacle,
but we need the usual dispersive estimates which are known to  be false when there exist conjugate points on the manifold. However we can recover this by following the argument in \cite{HZ}.

We first prove \eqref{l-dinh}.  Recall $\leftidx{^{\sigma}}U^{\mathrm{low}}(t)$ in \eqref{Ul}, as before, it suffices to show
\begin{equation}\label{l-dinh'}
\Big\|\int_{\tau<t} \leftidx{^{\sigma}}U^{\mathrm{low}}(t)(\leftidx{^{\sigma}}U^{\mathrm{low}}(\tau))^*F(\tau) d\tau\Big\|_{L^2_t(\R:L^r(M^\circ))}\lesssim \|F\|_{L^{2}_t(\R:L^{\tilde{r}'}(M^\circ))},\quad \sigma=1/2.
\end{equation}
To show \eqref{l-dinh'},  it is enough to show the bilinear form estimate
\begin{equation}\label{lTFG}
|T(F,G)|\lesssim  \|F\|_{L^{2}_tL^{\tilde{r}'}_z}\|G\|_{L^{2}_tL^{r'}_z},
\end{equation}
where $T(F,G)$ is the bilinear form
\begin{equation}
T(F,G)=\iint_{\tau<t}\langle \leftidx{^{\sigma}}U^{\mathrm{low}}(t)(\leftidx{^{\sigma}}U^{\mathrm{low}}(\tau))^*F(\tau), G(t)\rangle_{L^2}~ d\tau dt.
\end{equation}
Note that 
\begin{equation}\begin{split}
\leftidx{^{\sigma}}U^{\mathrm{low}}(t) (\leftidx{^{\sigma}}U^{\mathrm{low}}(\tau))^* 
&= \int_0^\infty e^{i(t-\tau)\lambda}\chi(\lambda) \lambda^{-2\sigma} dE_{\sqrt{H}}(\lambda)\\
&= \sum_{j\leq 0}\int_0^\infty e^{i(t-\tau)\lambda}\chi(\lambda) \varphi\big(
2^{-j}\lambda\big)  \lambda^{-2\sigma} dE_{\sqrt{H}}(\lambda).
\end{split}\end{equation}
Note that the summation term is close to \eqref{l-express}
\begin{equation}\begin{gathered}
\leftidx{^{\sigma}}U^{\mathrm{low}}_{j}(t) (\leftidx{^{\sigma}}U^{\mathrm{low}}_{j}(\tau))^* = \int_0^\infty e^{i(t-\tau)\lambda}\chi^2(\lambda) \varphi^2\big(
\frac{\lambda}{2^j} \big) \lambda^{-2\sigma} dE_{\sqrt{H}}(\lambda).
\end{gathered}\end{equation}
Therefore we can use the same argument to prove the same dispersive estimate \eqref{dispersivel'}.
Using \eqref{dispersivel'}  with positive sign, we obtain
\begin{equation*}
\begin{split}
&\langle \leftidx{^{\sigma}}U^{\mathrm{low}}(t)(\leftidx{^{\sigma}}U^{\mathrm{low}}(\tau))^*F(\tau), G(t)\rangle_{L^2}\\&\leq
C\sum_{j\leq0} 2^{\epsilon j}(1+|t-\tau|)^{2\sigma-3+\epsilon}\big\|\int e^{-(\rho-\delta)d(z,z')}F(\tau)dg(z')\big\|_{L^{r}}\|G(t)\|_{L^{r'}}
\\&\leq C(1+|t-\tau|)^{2\sigma-3+\epsilon}\|F(\tau)\|_{L^{\tilde{r}'}}\|G(t)\|_{L^{r'}}.
\end{split}
\end{equation*}
By using H\"older's and Young's inequalities and the fact $\sigma=1/2$ and $0<\epsilon\ll1$, we obtain 
\begin{equation*}
\begin{split}
|T(F,G)|&\lesssim
\iint_{\tau<t}(1+|t-\tau|)^{2\sigma-3+\epsilon}\|F(\tau)\|_{L^{r'}}\|G(t)\|_{L^{r'}}dtd\tau\\&
\lesssim \|F\|_{L^{2}_tL^{\tilde{r}'}}\|G\|_{L^{2}_tL^{r'}}.
\end{split}
\end{equation*}
This proves \eqref{lTFG}, and hence \eqref{l-dinh}.\vspace{0.2cm}

We next prove \eqref{h-dinh}.  Recall $\leftidx{^{\sigma}}U^{\mathrm{high}}(t)$ in \eqref{Uh}, as before, it suffices to show
\begin{equation}\label{h-dinh'}
\begin{split}
\Big\|\int_{\tau<t}\leftidx{^{\mu}}U^{\mathrm{high}}(t)(\leftidx{^{\tilde{\mu}}}U^{\mathrm{high}}(\tau))^*F(\tau)d\tau\Big\|_{L^2_tL^{r}_z}\lesssim \|F\|_{L^{2}_t{L}^{\tilde{r}'}_z}.
\end{split}
\end{equation}
To show \eqref{h-dinh'},  it is enough to show the bilinear form estimate
\begin{equation}\label{hTFG}
|T(F,G)|\lesssim  \|F\|_{L^{2}_tL^{\tilde{r}'}_z}\|G\|_{L^{2}_tL^{r'}_z},
\end{equation}
where $T(F,G)$ is the bilinear form
\begin{equation}
T(F,G)=\iint_{\tau<t}\langle \leftidx{^{\mu}}U^{\mathrm{high}}(t)(\leftidx{^{\tilde{\mu}}}U^{\mathrm{high}}(\tau))^*F(\tau), G(t)\rangle_{L^2}~ d\tau dt.
\end{equation}
Note that 
\begin{equation}\begin{split}
&\leftidx{^{\mu}}U^{\mathrm{high}}(t) (\leftidx{^{\tilde{\mu}}}U^{\mathrm{high}}(\tau))^*
\\&= \sum_{k,k'=0}^N \int_0^\infty e^{i(t-\tau)\lambda}(1-\chi)(\lambda) \lambda^{-(\mu+\tilde{\mu})} Q_k(\lambda) dE_{\sqrt{H}}(\lambda)Q_{k'}(\lambda)^*\\
&= \sum_{j\geq 0}\sum_{k,k'=0}^N \int_0^\infty e^{i(t-\tau)\lambda}(1-\chi)(\lambda) \varphi(2^{-j}\lambda)\lambda^{-(\mu+\tilde{\mu})} Q_k(\lambda) dE_{\sqrt{H}}(\lambda)Q_{k'}(\lambda)^* ,
\end{split}\end{equation}
in which the summation term is close to \eqref{h-express}
\begin{equation*}\begin{gathered}
\leftidx{^{\sigma}}U^{\mathrm{high}}_{j,k}(t) (\leftidx{^{\sigma}}U^{\mathrm{high}}_{j,k}(s))^* = \int_0^\infty e^{i(t-s)\lambda}(1-\chi)^2(\lambda) \varphi^2\big(
\frac{\lambda}{2^j} \big) \lambda^{-2\sigma}
Q_k(\lambda) dE_{\sqrt{H}}(\lambda) Q_k(\lambda)^* .
\end{gathered}\end{equation*}
The difference between the powers of functions $1-\chi$ and $\varphi$ is harmless. From Lemma \ref{poi} below, the case ``near-diagonal" ($k$ is close to $k'$) satisfies the same property of the case $k=k'$, thus it also leads to\eqref{Dispersiveh} and \eqref{Dispersiveh'}, hence it proves \eqref{hTFG}. 
In the case ``off diagonal"  in which the conjugate points are not separated, we can not prove the similar dispersive estimate like \eqref{Dispersiveh} and \eqref{Dispersiveh'}. However,  we can prove the following which  also leads to \eqref{hTFG}.

\begin{lemma}\label{either} Let $\leftidx{^{\sigma}}U_{\geq,k}^{\mathrm{high}}(t)$ be defined as in \eqref{U>}, then for each pair $(k,k')\in \{ 0, 1, \dots, N \}^2 $
there exists a constant $C$ such that, either
\begin{equation}\label{bilinear:s<t}
\iint_{\tau<t}\langle \leftidx{^{\mu}}U_{\geq,k}^{\mathrm{high}}(t) (\leftidx{^{\tilde{\mu}}}U_{\geq,k'}^{\mathrm{high}}(\tau))^* F(\tau), G(t)\rangle_{L^2}~ d\tau dt\leq C 
\|G\|_{L^2_\tau L^{r'}_z}\|F\|_{L^2_tL^{\tilde{r}'}_z},
\end{equation}
or
\begin{equation}\label{bilinear:s>t}
\iint_{\tau>t}\langle \leftidx{^{\mu}}U_{\geq,k}^{\mathrm{high}}(t) (\leftidx{^{\tilde{\mu}}}U_{\geq,k'}^{\mathrm{high}}(\tau))^*F(\tau), G(t)\rangle_{L^2}~ d\tau dt\leq C \|G\|_{L^2_\tau L^{r'}_z}\|F\|_{L^2_tL^{\tilde{r}'}_z}.
\end{equation}
\end{lemma}
We postpone the proof for a moment. Now we see how Lemma \ref{either} implies \eqref{hTFG}. On the one hand,  for every pair $(k,k')$, we have by \eqref{est:U>}
\begin{equation}\label{bilinear:U>}
\iint \langle \leftidx{^{\mu}}U_{\geq,k}^{\mathrm{high}}(t) (\leftidx{^{\tilde{\mu}}}U_{\geq,k'}^{\mathrm{high}}(\tau))^* F(\tau), G(t)\rangle_{L^2}~ d\tau dt\leq C 
\|G\|_{L^2_\tau L^{r'}_z}\|F\|_{L^2_tL^{\tilde{r}'}_z}.
\end{equation}
Hence for every pair $(k,k')$, 
 by \eqref{bilinear:s<t} or subtracting \eqref{bilinear:s>t} from \eqref{bilinear:U>}, we obtain
\begin{equation*}
\iint_{\tau<t}\langle \leftidx{^{\mu}}U_{\geq,k}^{\mathrm{high}}(t) (\leftidx{^{\tilde{\mu}}}U_{\geq,k'}^{\mathrm{high}}(\tau))^* F(\tau), G(t)\rangle_{L^2}~ d\tau dt\leq C 
\|G\|_{L^2_\tau L^{r'}_z}\|F\|_{L^2_tL^{\tilde{r}'}_z}.
\end{equation*}
Finally by summing over all $k$ and $k'$, we obtain \eqref{hTFG}. Once we prove Lemma \ref{either}, we complete the proof of Proposition \ref{prop:dinh}. 
\end{proof}

To prove Lemma \ref{either}, we need a result about the dispersive estimates.  To state and prove the dispersive estimates, we need
to categorize all microlocalization pairs $\{Q_k,Q_{k'}\}_{k,k'=0}^N$ and the property of spectral measure.

\begin{lemma}\label{poi}
The partition of the identity $Q_k(\lambda)$ can be chosen so that
the pairs of indices $(k,k')$, $1 \leq k,k' \leq N$, can be divided
into three classes,
$$
\{ 1, \dots, N \}^2 = J_{near} \cup J_{not-out} \cup J_{not-inc},
$$
so that
\begin{itemize}
\item if $(k,k') \in J_{near}$, then $Q_k(\lambda) \specl Q_{k'}(\lambda)^*$ satisfies the conclusions of Proposition~\ref{prop:localized spectral measure};

\item if $(k,k') \in J_{non-inc}$, then $Q_k(\lambda)$ is not incoming-related to $Q_{k'}(\lambda)$ in the sense that no point in the operator wavefront set (microlocal support)
of $Q_k(\lambda)$ is related to a point in the operator wavefront
set of $Q_{k'}(\lambda)$ by backward bicharacteristic flow;

\item if $(k,k') \in J_{non-out}$, then $Q_k(\lambda)$ is not outgoing-related to $Q_{k'}(\lambda)$ in the sense that no point in the operator wavefront set  of
$Q_k(\lambda)$ is related to a point in the operator wavefront set
of $Q_{k'}(\lambda)$ by forward bicharacteristic flow.
\end{itemize}
\end{lemma}

\begin{proof}
This is an analogue of \cite[Lemma 8.2]{HZ} which is stated in the asymptotically conic manifold.
The proof of the non-trapping asymptotically hyperbolic manifold is given in \cite{Chen} which is essentially due to \cite{GH}.

\end{proof}

Using the not-incoming or not-outgoing property of $Q_k(\lambda)$ with respect to $Q_{k'}(\lambda)$,
one obtain a similar lemma \cite[Lemma 8.5]{HZ} for spectral measure. We omit the details but we point out the key idea which
also was used in \cite{Chen} considering the endpoint inhomogeneous Strichartz estimate for Schr\"odinger on the same setting considered here.

The essential key point is that the phase function in the oscillation expression of the Schwartz kernel of  $Q_k(\lambda) \specl Q_{k'}(\lambda)^*$ has
an unchanged sign when $(k,k') \in J_{non-inc}$ or $(k,k') \in J_{non-out}$. More precisely, there exists a small constant $c>0$ such that the phase function $\Phi\geq c$ when $(k,k') \in J_{non-out}$
and $\Phi\leq -c$ when $(k,k') \in J_{non-inc}$. For simple, we only take one example to illustrate the idea.   If $Q_k$
is not incoming-related to $Q_{k'}$, we only consider 
\begin{equation*}
\begin{split}
Q_k(\lambda) dE_{\sqrt{H}}(\lambda)Q_{k'}(\lambda)^*=\int_{\R^m} e^{i\lambda\Phi(z,z',v)}\lambda^{n-1+\frac m2}a(\lambda,z,z',v) dv,
\end{split}
\end{equation*}
where $\Phi(z,z',v)\geq c>0$ and $|(\lambda\partial_\lambda)^\alpha a|\leq C_\alpha e^{-(\rho-\delta)d(z,z')}$ where $\rho=(n-1)/2$ and $0<\delta\ll1$.
Here the parameter $0\leq m\leq n-1$ is connected to the conjugate points which is the degenerate rank of the projection from the phase space to the base. If one review the previous result in \cite{HZ} and reference therein, one will find that $m=0$ if there is no
conjugate points in the manifold, then the expression will be similar to the case $k=k'$ in which the conjugate points are separated. If $m>0$, then it brings a difficulty in showing the 
dispersive estimate when $\lambda\to\infty$. However,  if  we restricted to $\tau<t$, then the microlocalized wave propagator 
\begin{equation*}
\begin{split}
\int_0^\infty e^{i(t-\tau)\lambda}\int_{\R^m} e^{i\lambda\Phi(z,z',v)}\lambda^{n-1+\frac m2}a(\lambda,z,z',v) dv d\lambda
\end{split}
\end{equation*}
has the phase function satisfying $(t-\tau)+\Phi\geq \max\{|t-\tau|,c\}$ due to the fact that $\Phi$ and $t-\tau$ have the same signs. Hence we can overcome the difficulties by integration by parts. More precisely, we shall prove that

\begin{proposition}\label{prop:Dispersive} 
Let $\rho=(n-1)/2$ and $0<\delta\ll1$. There exist a constant $C$ independent of $t, z, z'$ for all
$(k,k')\in \{0,1,\cdots, N\}^2, j\geq0$ such that
the following pointwise estimates  hold for any $K\geq 0$:

\begin{itemize}
\item If $k=0$ or $k'=0$ or  $(k,k') \in J_{near}$, then for all $t \neq \tau$  we have
\begin{equation}\label{kk'near}
\begin{split}
\Big|\int_0^\infty e^{i(t-\tau)\lambda}&(1-\chi)(\lambda) \varphi(2^{-j}\lambda)\lambda^{-(\mu+\tilde{\mu})} Q_k(\lambda) dE_{\sqrt{H}}(\lambda)Q_{k'}(\lambda)^*\Big|
\\&\lesssim \begin{cases} 
2^{j[\frac{n+1}2-(\mu+\tilde{\mu})]}(2^{-j}+|t-\tau|)^{-\frac{n-1}2}e^{-(\rho-\delta)d(z,z')}, \quad |t-\tau|\leq 2;\\
2^{j[\frac{n+1}2-(\mu+\tilde{\mu})]}|t-\tau|^{-K}e^{-(\rho-\delta)d(z,z')}, \quad |t-\tau|\geq 2.
\end{cases}
\end{split}
\end{equation}

\item If $(k,k')\in J_{non-out}$, that is, $Q_k$ is not outgoing related to
$Q_{k'}$, and $t<\tau$, then
\begin{equation}\label{kk'out}
\begin{split}
\Big|\int_0^\infty e^{i(t-\tau)\lambda}&(1-\chi)(\lambda) \varphi(2^{-j}\lambda)\lambda^{-(\mu+\tilde{\mu})} Q_k(\lambda) dE_{\sqrt{H}}(\lambda)Q_{k'}(\lambda)^*\Big|
\\&\lesssim \begin{cases} 
2^{j[\frac{n+1}2-(\mu+\tilde{\mu})]}(2^{-j}+|t-\tau|)^{-\frac{n-1}2}e^{-(\rho-\delta)d(z,z')}, \quad |t-\tau|\leq 2;\\
2^{j[\frac{n+1}2-(\mu+\tilde{\mu})]}|t-\tau|^{-K}e^{-(\rho-\delta)d(z,z')}, \quad |t-\tau|\geq 2.
\end{cases}
\end{split}
\end{equation}

\item Similarly, if $(k,k')\in J_{non-inc}$, that is, $Q_k$ is not incoming
related to $Q_{k'}$, and $\tau<t$, then
\begin{equation}\label{kk'in}
\begin{split}
\Big|\int_0^\infty e^{i(t-\tau)\lambda}&(1-\chi)(\lambda) \varphi(2^{-j}\lambda)\lambda^{-(\mu+\tilde{\mu})} Q_k(\lambda) dE_{\sqrt{H}}(\lambda)Q_{k'}(\lambda)^*\Big|
\\&\lesssim \begin{cases} 
2^{j[\frac{n+1}2-(\mu+\tilde{\mu})]}(2^{-j}+|t-\tau|)^{-\frac{n-1}2}e^{-(\rho-\delta)d(z,z')}, \quad |t-\tau|\leq 2;\\
2^{j[\frac{n+1}2-(\mu+\tilde{\mu})]}|t-\tau|^{-K}e^{-(\rho-\delta)d(z,z')}, \quad |t-\tau|\geq 2.
\end{cases}
\end{split}
\end{equation}
\end{itemize}

\end{proposition}

Now we prove Lemma~\ref{either} assuming Proposition \ref {prop:Dispersive}.

\begin{proof}[Proof of Lemma~\ref{either}] The main argument is to repeat the argument in the proof of Proposition \ref{m-strichartz} with $j\geq0$ due to \cite{KT}  if we have the dispersive estimate. In the case that $(k,k') \in J_{near}$, 
we have the dispersive estimate \eqref{kk'near}.
We repeat the argument in the proof of Proposition \ref{m-strichartz} and sum in $j\geq0$ to obtain
\eqref{bilinear:s<t}.  We would like to remark that $\mu,\tilde{\mu}>s$ ensures the summation in $j\geq0$ converges. If $(k,k') \in J_{non-inc}$, we
obtain \eqref{bilinear:s<t} 
due to the dispersive estimate \eqref{kk'in} when $\tau < t$.
Finally, in the case that $(k,k') \in J_{non-out}$, we obtain
\eqref{bilinear:s>t} since we have the dispersive estimate
\eqref{kk'out} for $\tau> t$. 

\end{proof}

\begin{proof} [Proof of  Proposition \ref{prop:Dispersive}] The proof is modified from the proof for Schr\"odinger equation in \cite[Lemma 8.6]{HZ} adapted to the wave equation.

We first prove \eqref{kk'near}. If one of $k, k'$ equals $0$, we have the expression of microlocalized spectral mearsue in Proposition \ref{prop:localized spectral measure} since
the support of $Q_0$ far away the boundary. From above result, if $(k,k') \in J_{near}$, by Lemma \ref{poi}, we also have the expression of microlocalized spectral mearsue in Proposition \ref{prop:localized spectral measure}.
Hence we can prove \eqref{kk'near} by using the same argument of proving \eqref{dispersiveh} and \eqref{dispersiveh'} in  Lemma \ref{lem:dispersive}. We omit the details here. \vspace{0.2cm}

We only prove \eqref{kk'in} since \eqref{kk'out} follows  from the same argument. Assume that $Q_k$
is not incoming-related to $Q_{k'}$. 
In this case, for the sake of simplicity, we only consider 
\begin{equation*}
\begin{split}
Q_k(\lambda) dE_{\sqrt{H}}(\lambda)Q_{k'}(\lambda)^*=\int_{\R^m} e^{i\lambda\Phi(z,z',v)}\lambda^{n-1+\frac m2}a(\lambda,z,z',v) dv,
\end{split}
\end{equation*}
where $\Phi(z,z',v)\geq \epsilon>0$, $0\leq m\leq n-1$ and $a$ is a smooth function which is compactly supported in the  $v$ such that 
$|(\lambda\partial_\lambda)^\alpha a|\leq C_\alpha e^{-(\rho-\delta)d(z,z')}$. For example, see \cite[(8-13), Lemma 8.5]{HZ}.
Then we need to show that for $\tau<t$ and $j\geq0$
\begin{equation}\label{kk'in'}
\begin{split}
\Big|\int_0^\infty e^{i(t-\tau)\lambda}&(1-\chi)(\lambda) \varphi(2^{-j}\lambda)\lambda^{-(\mu+\tilde{\mu})} \int_{\R^m} e^{i\lambda\Phi(z,z',v)}\lambda^{n-1+\frac m2}a(\lambda,z,z',v) dv d\lambda\Big|
\\&\lesssim \begin{cases} 
2^{j[\frac{n+1}2-(\mu+\tilde{\mu})]}(2^{-j}+|t-\tau|)^{-\frac{n-1}2}e^{-(\rho-\delta)d(z,z')}, \quad |t-\tau|\leq 2;\\
2^{j[\frac{n+1}2-(\mu+\tilde{\mu})]}|t-\tau|^{-K}e^{-(\rho-\delta)d(z,z')}, \quad |t-\tau|\geq 2.\\
\end{cases}
\end{split}
\end{equation}
Indeed, we can directly obtain by integration by parts
 \begin{equation*}
\begin{split}
&\Big|\int_0^\infty e^{i(t-\tau)\lambda}(1-\chi)(\lambda) \varphi(2^{-j}\lambda)\lambda^{-(\mu+\tilde{\mu})} \int_{\R^m} e^{i\lambda\Phi(z,z',v)}\lambda^{n-1+\frac m2}a(\lambda,z,z',v) dv d\lambda\Big|
\\&\lesssim\Big|\int_{\R^m} \int_0^\infty |(t-\tau)+\Phi(z,z',v)|^{-K}\partial_\lambda^K\left((1-\chi)(\lambda) \varphi(2^{-j}\lambda)\lambda^{-(\mu+\tilde{\mu})} \lambda^{n-1+\frac m2}a(\lambda,z,z',v) \right) d\lambda dv\Big|.
\end{split}
\end{equation*}
Note that $a$ is compactly supported in variable $v$, $t-\tau>0$ and $\Phi(z,z',v)\geq \epsilon>0$. Consequently,
 \begin{equation*}
\begin{split}
&\Big|\int_0^\infty e^{i(t-\tau)\lambda}(1-\chi)(\lambda) \varphi(2^{-j}\lambda)\lambda^{-(\mu+\tilde{\mu})} \int_{\R^m} e^{i\lambda\Phi(z,z',v)}\lambda^{n-1+\frac m2}a(\lambda,z,z',v) dv d\lambda\Big|\\
&\lesssim|(t-\tau)+\epsilon|^{-K} \int_{2^{j-1}}^{2^{j+1}}  \lambda^{n-1+\frac m2-\mu-\tilde{\mu}-K} d\lambda \,e^{-(\rho-\delta)d(z,z')}
\\&\lesssim 2^{j(n+\frac m2-\mu-\tilde{\mu})}(2^j(|t-\tau|+\epsilon))^{-K}e^{-(\rho-\delta)d(z,z')},
\end{split}
\end{equation*}
which implies \eqref{kk'in'} by choosing $K$ large enough.

\end{proof}

\section{Proof of Theorem \ref{GWP}}

This section is devoted to the proof of Theorem \ref{GWP} by using the Strichartz estimates in Theorem \ref{thm:hstri} and Theorem \ref{thm:ihstri}.\vspace{0.2cm}

Let $p >1$. The proof is standard and based on a contraction mapping argument in the Banach space $L^{p+1}(\mathbb R^+ \times M^\circ)$. Define the map $\mathcal T$ by $v=\mathcal T u$ where $v$ solves, given $u \in L^{p+1}(\mathbb R^+ \times M^\circ)$, 
\begin{equation}
\begin{cases}
\partial_{t}^2v-\Delta_g v-\rho^2 v=F_p(u), \quad (t,z)\in I\times M^\circ; \\ u(0)=\nu u_0(z),
~\partial_tu(0)=\nu u_1(z).
\end{cases}
\end{equation}
Notice first that all the Strichartz estimates are global in time so one has $I=\mathbb R^+$.  
Choose $q=r=\tilde q=\tilde r=p+1>2$, then we can verify, for any $p \in (1, 1+\frac{4}{n-1})$,
\begin{equation*}
(p+1,p+1,\mu)\in \Lambda_e,\quad s_e< \mu
\end{equation*}
and
\begin{equation*}
(p+1,p+1,1/2)\in \Lambda_e,\quad s_e< 1/2.
\end{equation*}
Therefore, for fixing $0<\epsilon\ll1$, we can apply Theorems \ref{thm:hstri} with $(p+1,p+1,\mu_0)\in \Lambda_e$ and Theorem \ref{thm:ihstri} with $(p+1,p+1,1/2)\in \Lambda_e$ (or directly \eqref{(p+1)-inh}) to obtain
 \begin{equation*}
\begin{split}
\|v(t,z)\|_{L^{p+1}(I \times M^\circ)}
\lesssim \nu \Big (\|u_0\|_{H^{\mu,0}(M^\circ)}&+\|u_1\|_{
H^{\mu-1,-\epsilon}(M^\circ)} \Big )+\||u (t,z)|^p\|_{L^{\frac{p+1}{p}}_t(I;L^{\frac{p+1}{p}}(M^\circ))}.
\end{split}
\end{equation*}
Thus this gives 
 \begin{equation*}
\begin{split}
&\|v(t,z)\|_{L^{p+1}(I \times M^\circ)}
\lesssim \nu \Big (\|u_0\|_{H^{\mu,0}(M^\circ)}+\|u_1\|_{
H^{\mu-1,-\epsilon}(M^\circ)} \Big )+ \|u\|^{p}_{L^{p+1}_t(I;L^{p+1}(M^\circ))}.
\end{split}
\end{equation*}

Therefore the operator $\mathcal T$ maps $L^{p+1}(\mathbb R^+ \times M^\circ)$ into itself. Furthermore, a standard computation shows that if $\nu$ is small enough, $\mathcal T$ maps a ball of $L^{p+1}(\mathbb R^+ \times M^\circ)$ into itself and is actually a contraction, hence by the Banach fixed point theorem this leads to the desired result (see for instance \cite{SSW}).

\begin{center}

\end{center}

\end{document}